\documentclass[10pt]{article}
\usepackage{amsmath, amsthm,amssymb}
\usepackage{bm}
\usepackage{bbm}
\usepackage{xcolor}
\usepackage[hidelinks]{hyperref} 
\usepackage{makeidx}
\usepackage{enumerate}
\usepackage{stmaryrd}
\numberwithin{equation}{subsection}
\date{}

\newtheorem{thm}{Theorem}[section]

\newtheorem{prop}[thm]{Proposition}
\newtheorem{cor}[thm]{Corollary}
\newtheorem{lem}[thm]{Lemma}

\newtheorem{claim}[thm]{Claim}

\newtheorem{prob}[thm]{Problem}

\newtheorem{question}[thm]{Question}

\newcommand{\tp}{{\bar\otimes}}
\newcommand{\mr}{\mathcal R}

\newtheorem{defn}[thm]{Definition}

\newcommand{\ca}{\curvearrowright}
\newcommand{\G}{\Gamma}
\newcommand{\La}{\Lambda}
\newcommand{\g}{\gamma}
\newcommand{\la}{\lambda}
\newcommand{\de}{\delta}

\newcommand{\ra}{{\rightarrow}}

\newcommand{\mc}{\mathbb{C}}
\newcommand{\email}{Email: } 
\begin{document}

\title{Rigidity results for von Neumann algebras arising from mixing extensions of profinite actions of groups on probability spaces}
\author{Ionut Chifan and Sayan Das}

\maketitle

\begin{abstract} \noindent  Motivated by Popa's seminal work \cite{Po04}, in this paper, we provide a fairly large class of examples of group actions $\G \ca X$ satisfying the extended Neshveyev-St\o rmer rigidity phenomenon \cite{NS03}: whenever $\La \ca Y$ is a free ergodic pmp action and there is a $\ast$-isomorphism $\Theta:L^\infty(X)\rtimes \G \ra L^\infty(Y)\rtimes \La$ such that $\Theta(L(\G))=L(\La)$ then the actions $\G\ca X$ and $\La \ca Y$ are conjugate (in a way compatible with $\Theta$). We also obtain a complete description of the intermediate subalgebras of all (possibly non-free) compact extensions of group actions in the same spirit as the recent results of Suzuki \cite{Suzuki}. This yields new consequences to the study of rigidity for crossed product von Neumann algebras and to the classification of subfactors of finite Jones index.     
\end{abstract}

\section{Introduction}
In the mid thirties Murray and von Neumann found a natural way to associate a von Neumann algebra to any measure preserving action $\G \ca X$ of a countable group $\G$ on a probability space $X$.  This is called the group measure space von Neumann algebra, denoted by $L^{\infty}(X) \rtimes \G$. The most interesting case for study is when the initial action $\G \ca X$ is free and ergodic, in which case the group measure space construction is in fact a type $\rm II_1$ factor. When $X$ is a singleton the group measure space construction yields just the group von Neumann algebra that will be denoted by $L(\G)$. The latter is a II$_1$  factor specifically when all nontrivial conjugacy classes of $\G$ are infinite (henceforth abbreviated as the icc property).   
\vskip 0.05in

\noindent A problem of central importance in von Neumann algebras is to determine how much information about the action $\G \ca X$ can be recovered from the isomorphism class of $L^\infty(X)\rtimes \G$. An unprecedented progress in this direction emerged over the last decade from Popa's influential deformation/rigidity theory \cite{Po06}. A remarkable achievement of this theory was the discovery of first classes of examples of actions that are entirely remembered by their von Neumann algebras; for some examples see \cite{Po06,Po07,Io08,Pe09,PV09, Io10,CP10,Va10,HPV10,FV10,CS11,CSU11,PV11,PV12,Io12,Bo12,CIK13,CK15,Dr16,GIT-D16}. We refer the reader to the surveys \cite{Vaicm, Io18} for an overview of the recent developments.

\vskip 0.05in

%\noindent Associated to a free, ergodic, probability measure preserving (henceforth abbreviated as p.m.p) action of a countable, discrete group $\G$ on a probability measure space $(X, \mu)$ one obtains two \textit{canonical} von Neumann algebras $L^{\infty}(X) \rtimes \G$ and $L(\G)$ via two constructions of Murray and von Neumann. The von Neumann algebra $L^{\infty}(X) \rtimes \G$ is the so called \textit{group measure space} construction, while the algebra $L(\G)$ is called the group von Neumann algebra. Under the hypothesis of freeness and ergodicity, it is well known that the group measure space construction is a factor, which is a type $\rm II_1$ factor if the action is p.m.p, while $L(\G)$ is a finite von Neumann algebra, which is a type $\rm II_1$ factor if and only if $\G$ is an i.c.c.\ group. 

\noindent There are two distinguished subalgebras of $L^\infty(X)\rtimes \G$: the coefficient (or Cartan) subalgebra $L^\infty(X)\subset L^\infty(X)\rtimes \G$ and the group von Neumann subalgebra $L(\G)\subset L^\infty(X)\rtimes \G$. The classification of group measure space von Neumann algebra is closely related to the study of these two inclusions of von Neumann algebras. For instance, in \cite{Si55} Singer observed that the study of the inclusion $L^{\infty}(X) \subset L^{\infty}(X) \rtimes \G$ amounts to the study of the equivalence relation induced by the orbits of $\G\ca X$. Thus reconstructing the action $\G \ca X$ from the inclusion $L^\infty(X)\subset L^\infty(X)\rtimes \G$ relies upon the reconstruction from its orbits. This theme in contemporary ergodic theory is known as orbit equivalence rigidity. The study of orbit equivalence rigidity has received a lot of attention over the last couple of decades and has major consequences to the classification of von Neumann algebras in general, and the structure of the crossed product algebras in particular; for instance see \cite{Fu99a,Ga09,MS04,Ki06,Io08,CK15,GIT-D16}.  

\noindent Deriving information about the action $\G \ca X$ from the other inclusion $L(\G)\subset L^\infty(X)\rtimes \G$ is another topic which is implicit in many core rigidity results in von Neumann algebras \cite{NS03,Po03,Po04,OP07}.  When $\G$ is abelian $L^\infty(X)\rtimes \G=\mathcal R$ is the hyperfinite $\rm II_1$ factor and each of $L^{\infty}(X)$ and $L(\G)$ is a maximal abelian subalgebra of $\mathcal R$ (henceforth abbreviated as MASA). In their study on structural aspects of these MASAs in \cite{NS03} Neshveyev and St\o rmer discovered that the positions of these two MASAs inside $\mathcal R$ completely determines the action. More precisely, they showed the following: \emph{Let $\G$ be an infinite abelian group, $\G \ca X$ be a weak mixing action and $\La \ca Y$ be any action. If there is a $\ast$-isomorphism $\Theta: L^\infty(X)\rtimes \G \ra L^\infty(Y)\rtimes \La$ satisfying $\Theta(L(\G))= L(\La)$ and $\Theta(L^\infty(X))$ is inner conjugate to $L^\infty(Y)$ then $\G\ca X$ is conjugate with $\La\ca Y$ (in a way compatible to $\Theta$).} They also conjectured the same statement holds without the inner conjugacy of the Cartan subalgebras condition. In other words the inclusion $L(\G) \subset L^{\infty}(X) \rtimes \G$ alone completely captures the entire crossed product structure of $L^\infty(X)\rtimes \G$. 
\noindent The first examples of actions satisfying the full statement of Neshveyev-St\o rmer  conjecture emerged from the impressive work of Popa on the classification of von Neumann algebras associated with Bernoulli actions, \cite{Po03,Po04}. Specifically, using his influential deformation/rigidity theory Popa was able to show that this is the case for all clustering (e.g.\ Bernoulli) actions $\G\ca X$ \cite[Theorem 0.7]{Po04}. Remarkably, this holds even when $\G$ is nonabelian. These significant initial advances strongly suggest that the Neshveyev-Stormer conjecture could hold in a much larger generality that supersedes the amenable regime (e.g.\ $\G$ is abelian).  Motivated by this and the implicit relevance to the study of rigidity aspects for crossed products it is natural to investigate the following extended version of the Neshveyev-St\o rmer rigidity question: 
\begin{question}[Extended Neshveyev-Stormer rigidity question]\label{ENSconj}
	Let $\G$ and $\La$ be icc countable discrete groups and let $\G \ca X$ and $\La \ca Y$ be free, ergodic, pmp actions. Assume that there is a $\ast$-isomorphism $\Theta: L^{\infty}(X) \rtimes \G \rightarrow L^{\infty}(Y) \rtimes \La$ such that $\Theta(L(\G))=L(\La)$. Under what conditions on $\G \ca X$ are the actions $\G \ca X$ and $\La \ca Y$ conjugate?
\end{question}

\noindent Besides Popa's examples at this time there are several other families of specific actions $\G\ca X$ for which Question \ref{ENSconj} has a solution. These arise mostly from  decade-long developments in the classification of von Neumann algebras via Popa's deformation/rigidity program. For instance, this is the case for all  $W^*$-superrigid actions
(see \cite{Io18} for a survey on $W^{\ast}$ superrigidity and the references therein). Also, using \cite[Theorem 5.2]{Po04} one can easily see that the rigidity phenomenon in Question \ref{ENSconj} is also satisfied by any weak mixing action $\G \ca X$ for which, up to unitary conjugacy, $L^\infty(X)$ is the unique group measure space Cartan subalgebra of $L^\infty(X)\rtimes \G$. This way one can get more examples using the recent results on uniqueness of Cartan subalgebras, see \cite{OP07,PV11,PV12,Io12,CIK13,CK15} for example. However not much was known beyond these classes of examples and it remained open to find a more intrinsic approach to Question \ref{ENSconj} which does not rely on uniqueness of Cartan subalgebras results from deformation/rigidity theory. 
\vskip 0.05in
\noindent In this article we develop new technical aspects that enables us to partially answer Question \ref{ENSconj}. In particular we are able to describe a fairly large family of actions which covers many new examples beyond all the aforementioned classes, e.g.\ all nontrivial mixing extensions of free compact actions, satisfying the extended Neshveyev-St\o rmer rigidity phenomenon. More generally, we have the following result.

\begin{thm}\label{main2intro}
	Let $\G$ be an icc group and let $\Gamma\ca^\sigma X$ be an action whose distal quotient $\G \ca X_d $ is free and the extension $\pi: X \rightarrow X_d$ is (nontrivial) mixing. Let $\Lambda\ca^\alpha Y$ be any action. Assume that $\Theta: L^\infty(X)\rtimes \G \ra L^\infty(Y)\rtimes \Lambda$ is a $\ast$-isomorphism such that $\Theta(L(\G))=L(\La)$. Then there exist a unitary $x \in L(\La)$, a character $\omega: \G \rightarrow \mathbb{T}$, and a group isomorphism $\delta: \G \rightarrow \La$ such that $x\Theta(L^{\infty}(X))x^{\ast}=L^{\infty}(Y)$ and for all $a \in L^{\infty}(X), \g \in \G$ we have $$\Theta(au_{\g})=\omega(\g)\Theta(a) x^{\ast}v_{\delta(\g)}x.$$ In particular, we have $x \Theta (\sigma_\g(a))x^*= \alpha_{\delta(\g)}(x\Theta(a)x^*)$ and hence $\G\ca X$ and $\La\ca Y$ are conjugate. 
\vskip 0.01in	
	\noindent Here $\{u_\g\}_{\g\in \G}$  and $\{v_\la \}_{\la\in \La}$ are the canonical group unitaries implementing the actions in $L^\infty(X)\rtimes \G$ and $L^\infty(Y)\rtimes \La$, respectively.
\end{thm}

\noindent In particular the theorem implies that if $\G$ is any icc group then any action  $\Gamma\ca X$ which admits a free profinite quotient $\G \ca X_d $ with (nontrivial) mixing extension $\pi: X \rightarrow X_d$ satisfies the extended Neshveyev-St\o rmer rigidity question. As a concrete example let $\G$ be any icc residually finite group and let $\cdots \lhd \G_n\lhd \cdots \lhd  \G_2\lhd \G_1\lhd \G$ be a resolution of finite index normal subgroups satisfying $\cap_n \G_n = 1$. Consider the action $\G \ca (\G/\G_n, c_n)$ by left multiplication of $\G$ on the left cosets $\G/\G_n$ seen as a finite probability space with the counting measure $c_n$ and let $\G\ca (Z,\mu)= \varprojlim (\G/\G_n, c_n)$ be the inverse limit of these actions. In addition let $\pi:\G \ca \mathcal O(\mathcal H)$ be any mixing orthogonal representation and let $\G\ca (Y^\pi, \nu^\pi)$ be the corresponding Gaussian action. Then the diagonal action $\G \ca (Y^\pi \times Z, \nu^\pi \times \mu)$ is profinite-by-(nontrivial) mixing, and hence by Theorem \ref{main2intro} the rigidity Question \ref{ENSconj} has a positive solution in this case. 

\vskip 0.05in

\noindent Theorem \ref{main2intro} is obtained by heavily exploiting, at the von Neumann algebraic level, the natural tension that occurs between mixing and compactness properties for actions. Briefly, let $\G\ca X$ and $\La \ca Y$ be actions as in Theorem \ref{main2intro} so that $L^\infty(X)\rtimes \G =L^\infty(Y)\rtimes \La$ with $L(\G)=L(\La)$. First we use the description of compactness via quasinormalizers from \cite{CP18,Io08a} to identify the von Neumann algebras of their distal parts, i.e.\  $L^\infty(X_d)\rtimes \G= L^\infty(Y_d)\rtimes \La$. In turn this is used to show that the mixing property of the extension $L^\infty(X_d)\subseteq L^\infty(X)$ is transferred through von Neumann equivalence to the extension  $L^\infty(Y_d)\subseteq L^\infty(Y)$ (Theorem \ref{mix2}). Once these are established, some basic adaptations of Popa's intertwining techniques from \cite{Po03} further show that the Cartan subalgebras $L^\infty(X)$ and $L^\infty(Y)$ are in fact unitarily conjugate. Then the desired result is derived from a general principle which states that for any free ergodic actions $\G\ca X$, $\La\ca Y$ of icc groups $\G$ and $\La$, inner conjugacy of $L^\infty(X)$ and $L^\infty(Y)$ together with $L(\G)=L(\La)$  imply conjugacy of $\G\ca X$ and $\La\ca Y$ (Theorem \ref{3'}). This criterion for conjugacy of group actions generalizes the earlier works \cite{NS03,Po04} and is obtained using the notion of height of elements with respect to groups from \cite{IPV10}. Specifically, using Dye's theorem and an averaging argument we show that $\G$ has large height with respect to $\La$ inside $L(\La)$ (Theorem \ref{2'}). By \cite[Theorem 3.1]{IPV10} this further implies $\G$ is unitarily conjugate to $\La$. Further exploiting the icc condition we deduce conjugacy of the actions (Theorem \ref{3'}).       

\vskip 0.05in
\noindent While Theorem \ref{main2intro} settles the extended Neshveyev-St\o rmer rigidity question for nontrivial extensions, two natural extreme situations, namely, when $\G \ca X$ is either mixing or compact (even profinite) remain open. We believe that in both of these cases one should still get a positive answer and we formulate a few sub problems in this direction; see for instance Problem \ref{profnsprob}. However, in order to successfully tackle these questions, significant new technical advancements are needed. Specifically, if one pursues an approach similar to Theorem \ref{main2intro} the key step is to establish the inner conjugacy of  $L^\infty(X)$ and $L^\infty(Y)$. In the presence of mixing this would follow if one can show there exist free factors $\G \ca X_0$ of $\G \ca X$ and $\La \ca Y_0$ of $\La\ca Y$ whose von Neumann algebras coincide, i.e.\ $L^\infty(X_0)\rtimes \G = L^\infty(Y_0)\rtimes \La$; see Corollary \ref{mixingint}. In turn this highlights the importance of studying intermediate subalgebras in the inclusion $L(\G)\subset L^\infty(X)\rtimes \G$. In addition this seems relevant even to the study of Question \ref{ENSconj} for profinite actions.

\vskip 0.05in
\noindent  Note that when $\G$ is icc, and $\G \ca X$ is free, ergodic and pmp, the inclusion $L(\G) \subset L^{\infty}(X) \rtimes \G$ is an irreducible inclusion of $\rm II_1$ factors. In his seminal paper \cite{Jo81} Jones pioneered the study of inclusions of type $\rm II_1$ factors, or \textit{subfactors}. Subfactor theory has had a number of striking applications over the years in various diverse branches of mathematics and mathematical physics, including Knot theory and Conformal Field theory, \cite{Jo90,Jo91,Jo09}. A major motivating question in Subfactor theory is the classification of all intermediate subalgebras. Pursuing this perspective, we were able to classify all the intermediate subalgebras in compact extensions in the same spirit as Suzuki's recent results from \cite{Suzuki}. To properly introduce our result we briefly recall some terminology. Given two actions $\G \ca^\beta X_0$ and $\G \ca^\alpha X$ we say that $\alpha$ is an extension of $\beta$ if there is a $\G$-equivariant factor map $\pi: X\ra X_0$. At the von Neumann algebra level this induces an inclusion $L^\infty(X_0)\subseteq L^\infty(X)$ on which $\G$ acts naturally via $\alpha_\g(f)=f\circ \alpha_{\g^{-1}}$ when $f\in L^\infty(X)$.  An \emph{intermediate extension for $\pi$} (or \emph{between $\G \ca X_0$ and $\G\ca X$}) is an action $\G\ca Z$ for which there exist $\G$-equivariant factor maps $\pi_1: X\ra Z$ and $\pi_2:Z \ra X_0$ such that $\pi_2\circ\pi_1=\pi$. Note that the intermediate extensions of $\pi$  are in bijective correspondence with the $\G$-invariant intermediate subalgebras of $L^\infty(X_0)\subseteq L^\infty(X)$.  We show that there is a bijective correspondence between intermediate von Neumann algebras in crossed products and intermediate extensions of dynamical systems. More precisely, we have the following    

\begin{thm} \label{main3} Let $\G$ be an icc group and let $\G\ca^\beta X_0$ be a pmp action. Let $\G\ca X$ be an ergodic compact extension of $\beta$, \cite{Fur77}. Consider the corresponding group measure space von Neumann algebras and note that we have the following inclusion $L^\infty(X_0)\rtimes \G\subseteq  L^\infty(X)\rtimes \G$. Then for any intermediate von Neumann subalgebra $L^\infty(X_0)\rtimes \Gamma \subseteq N \subseteq L^\infty(X) \rtimes \Gamma $ there exists an intermediate extension $\G\ca Z$ between $\G \ca X$ and $\G \ca X_0$ satisfying $N= L^\infty(Z) \rtimes \G.$
\end{thm}

\noindent In many respects this theorem complements the results from \cite{Suzuki}; for instance, it covers various examples of non-free extensions, most notably, when $X_0$ is a singleton. In this situation our result provides a complete description of all intermediate von Neumann subalgebras in the inclusion $L(\G)\subseteq L^{\infty}(X) \rtimes \G$ for any compact ergodic action $\G \ca X$ of any icc group $\G$. This in turn yields new interesting consequences towards the classification of finite index subfactors. For example, combining Theorem \ref{main3} with the characterization of compactness via quasinormalizers from \cite[Theorem 6.10]{Io08a}, for \emph{any} icc group $\G$ and \emph{any} free ergodic action $\G \ca X$,  we are able to classify \emph{all} the intermediate subfactors $L(\G)\subseteq N\subseteq L^\infty(X)\rtimes \G$ with finite Jones index $[N:L(\G)]<\infty$. Specifically we show that all such $N$ could arise only from the transitive finite factors of $\G \ca X$ (see part 2.\ in Corollary \ref{intsubgenmain}); in particular, this entails that the Jones index $[N:L(\G)]$ is always a positive integer. This should be compared with the similar statement \cite[Corollary 2.4]{Po85} for the intermediate subfactors of the Cartan inclusion $L^\infty(X)\subset N \subseteq L^\infty(X)\rtimes \G$ with $[L^\infty(X)\rtimes \G : N]<\infty$.       
\vskip 0.03in

\begin{cor} \label{intsubgenmain}
Let $\Gamma$ be an icc group and let $\G \ca X$ be a free ergodic pmp action. If $M=L^\infty(X)\rtimes \G$ is the corresponding group measure space construction then the following hold:

\begin{enumerate} 
\item  For any intermediate von Neumann algebra $L(\Gamma) \subseteq N \subseteq L^{\infty}(X) \rtimes \Gamma $ satisfying $N \subseteq \mathcal {QN}_{M}(L(\Gamma))''$ there exists a factor $\G\ca X_0$ of $\G\ca X$ such that $N =L^\infty(X_0)\rtimes \G$.  
\item If $L(\Gamma) \subseteq N \subseteq L^{\infty}(X) \rtimes \Gamma $ is an intermediate subfactor with $[N: L(\Gamma)]< \infty$ then there is a finite, transitive factor $\G \ca X_0$ of $\G \ca X$ such that $N=L^\infty(X_0)\rtimes \G$; in particular, $[N:L(\Gamma)] \in  \mathbb{N}$. Thus for any subfactors $L(\Gamma) \subseteq N_1 \subseteq N_2 \subseteq L^{\infty}(X) \rtimes \Gamma $, with either $[N_1:L(\G)] < \infty$ or $\G \ca X$ compact,  we have $[N_2:N_1] \in \mathbb{N}\cup\{\infty\}.$
%\item Let $\G$ be an icc group with no proper finite index subgroups (eg simple) and let $\G \ca (X,\mu)$ be a free, ergodic, pmp action. Then, there are no nontrivial intermediate subfactors $L(\G)\subseteq N \subseteq L^{\infty}(X, \mu) \rtimes \G$ with $[N:L(\G)]<\infty$.
\end{enumerate}
\end{cor}

\noindent In particular, part 2.\ implies that for any icc group $\G$ with no proper finite index subgroups and any free ergodic action $\G\ca X$ there are no nontrivial intermediate subfactors $L(\G)\subseteq N \subseteq L^\infty(X)\rtimes \G$ of finite index $[N:L(\G)]<\infty$. For example this is the case for all $\G$ infinite simple groups, e.g.\ Tarski's monsters, Burger-Mozes groups \cite{BM01}, Camm's groups \cite{Ca53}, or Bhattacharjee's groups \cite{Bh94}, just to enumerate a few. 

\vskip 0.05in 
\noindent We point out in passing that Theorem \ref{main3} actually holds in a more general setting, namely, for actions of groups on compact extensions of possibly non-abelian von Neumann algebras; this notion is highlighted in Definition \ref{compact}. In this generality our result yields a twisted version of Ge's splitting theorem for tensor products (see Corollary \ref{twisted2}) in the same spirit as \cite[Example 4.14]{Suzuki}. 
\vskip 0.03in
\noindent The classification of the intermediate subalgebras in Theorem \ref{main3} is achieved through a new mix of analytic and algebraic techniques that combines factoriality arguments together with a general algebraic criterion outlined in Theorem \ref{usefulresult}.  We also note the same criterion can be used in conjunction with various soft analytical arguments to successfully recover, in the finite von Neumann algebra case, several well-known results such as Ge's tensor splitting theorem \cite[Theorem 3.1]{Ge} or the Galois correspondence for group actions \cite{Ch}. These applications are presented in Corollary \ref{tensorint} and Theorems \ref{twistedGe} and \ref{galoiscorr}.  
\vskip 0.05in
\noindent Finally, Theorem \ref{main3} in combination with methods from Popa's deformation/rigidity theory and Jones' finite index subfactor theory provide new insight towards rigidity aspects for II$_1$ factors arising from profinite actions $\G \ca  X$ of icc property (T) groups $\G$. While Ioana has already established in \cite{Io08} that such actions are completely reconstructible from their orbits, significantly less is known about their rigid behavior at the von Neumann algebraic level. When $\G$ is in addition properly proximal, Boutonnet, Ioana and Peterson showed in \cite{BIP18} using boundary techniques \cite{BC14} that all compact Cartan subalgebras in $L^\infty(X)\rtimes \G$ are unitarily conjugate to $L^\infty(X)$. (For $\G$ direct products of nonamenable biexact groups this already follows from the earlier works \cite{CS11,CSU11}.) Consequently, this combined with \cite{Io08} yields that for any non-commensurable groups $\G$ and $\La$ and any free ergodic profinite actions $\G\ca X$ and $\La \ca Y$ the von Neumann algebras $L^\infty(X)\rtimes \G$ and $L^\infty(Y)\rtimes \La$ are not isomorphic; remarkably, this is the case for lattices $\G= {\rm PSL}_n(Z)$ and $\La = {\rm PSL}_m(\mathbb Z)$ for all $n\neq m$.   However, without these additional assumption on $\G$, the study of von Neumann algebraic rigidity aspects for profinite (or compact) actions $\G\ca X$ remains an wide open problem. For example, even establishing strong rigidity results similar to the ones obtained in \cite{Po04} by Popa for Bernoulli actions of rigid groups seems elusive at this time. While it is very plausible that such results should hold true, we only have the following partial result at this time in this direction.
 
\begin{thm}\label{virtualmain4} Let $\G$ and $\La$ be icc property (T) groups. Let $\G \ca X = \varprojlim X_n$ be a free ergodic profinite action and let $\La \ca Y$ be a free ergodic compact action. Assume that $\Theta: L^\infty(X)\rtimes \G \ra L^\infty(Y)\rtimes \La$ is a $\ast$-isomorphism. Then $\La \ca Y=\varprojlim Y_n$ is also a profinite action. Moreover, there exist $l\in \mathbb N$ and a unitary $w\in L^\infty(Y)\rtimes \La$ such that $\Theta(L^\infty(X_{k+l})\rtimes \G)= w(L^\infty(Y_{k+1})\rtimes \La)w^*$ for every integer $k\geq 0$. 
	%In addition there exist finite index subgroups $\G_0\leqslant \G$ and $\La_0 \leqslant \La$ that admit finite index resolutions $...\lhd \G_n \lhd \G_{n-1}\lhd ...\lhd \G_1 \lhd\G_0$ and $...\lhd \La_n \lhd \La_{n-1}\lhd ...\lhd \La_1 \lhd\La_0$ such that $|\G_0/\G_n|= |\La/\La_0|$ for all $n\geq 1$.
\end{thm}

\noindent This should be compared with Popa's work on inductive limits of II$_1$ factors \cite{Po12}. Finally, the same strategy used in the proof of Theorem \ref{main2intro} can be successfully used in combination with  Theorem \ref{virtualmain4} to provide a purely von Neumann algebraic approach to a version of Ioana's orbit equivalence superrigidity theorem from \cite{Io08}; see the proof of Theorem \ref{ioanaoe}.

%%%%%%%%%%%%%%%%%%%%%%%%%%%%%%%%%%%%%%%%%%%%%%%%%%%%%%%%%%%%%%%%%%%%%%%%%%%%%%%%%%%%%%%%%%%%%%%%%%%%%%%%%%%%%%%%%%%
%%%%%%%%%%%%%%%%%%%%%%%%%%%%%%%%%%%%%%%%%%%%%%%%%%%%%%%%%%%%%%%%%%%%%%%%%%%%%%%%%%%%%%%%%%%%%%%%%%%%%%%%%%%%%%%%%%%
%%%%%%%%%%%%%%%%%%%%%%                        PRELIMINARIES                     %%%%%%%%%%%%%%%%%%%%%%%%%%%%%%%%%%%
%%%%%%%%%%%%%%%%%%%%%%%%%%%%%%%%%%%%%%%%%%%%%%%%%%%%%%%%%%%%%%%%%%%%%%%%%%%%%%%%%%%%%%%%%%%%%%%%%%%%%%%%%%%%%%%%%%%
%%%%%%%%%%%%%%%%%%%%%%%%%%%%%%%%%%%%%%%%%%%%%%%%%%%%%%%%%%%%%%%%%%%%%%%%%%%%%%%%%%%%%%%%%%%%%%%%%%%%%%%%%%%%%%%%%%%

\section{Some preliminaries and technical results}\label{prelim}

%%%%%%%%%%%%%%%%%%%%%%%%%%%%%%%%%%%%%%%%%%%%%%%%%%%%%%%%%%%%%%%%%%%%%%%%%%%%%%%%%%%%%%%%%%%%%%%%%%%%%%%%%%%%%%%%%%%
%%%%%%%%%%%%%%%%%%%%%%%%%%%%%%%%%%%%%%%%%%%%%%%%%%%%%%%%%%%%%%%%%%%%%%%%%%%%%%%%%%%%%%%%%%%%%%%%%%%%%%%%%%%%%%%%%%%
%%%%%%%%%%%%%%%%%%%%%%                POPA INTERTWINING TECHNIQUES              %%%%%%%%%%%%%%%%%%%%%%%%%%%%%%%%%%%
%%%%%%%%%%%%%%%%%%%%%%%%%%%%%%%%%%%%%%%%%%%%%%%%%%%%%%%%%%%%%%%%%%%%%%%%%%%%%%%%%%%%%%%%%%%%%%%%%%%%%%%%%%%%%%%%%%%
%%%%%%%%%%%%%%%%%%%%%%%%%%%%%%%%%%%%%%%%%%%%%%%%%%%%%%%%%%%%%%%%%%%%%%%%%%%%%%%%%%%%%%%%%%%%%%%%%%%%%%%%%%%%%%%%%%%

\subsection{Popa's intertwining techniques}
Over a decade ago, Popa introduced  in \cite [Theorem 2.1 and Corollary 2.3]{Po03} a powerful analytic criterion for identifying intertwiners between arbitrary subalgebras of tracial von Neumann algebras. This is now termed \emph{Popa's intertwining-by-bimodules technique}.

\begin {thm}\cite{Po03} \label{corner} Let $(M,\tau)$ be a separable tracial von Neumann algebra and let $P, Q\subseteq M$ be (not necessarily unital) von Neumann subalgebras. 
Then the following are equivalent:
\begin{enumerate}
\item There exist $ p\in  \mathcal P(P), q\in  \mathcal P(Q)$, a $\ast$-homomorphism $\theta:p P p\rightarrow qQ q$  and a partial isometry $0\neq v\in q M p$ such that $\theta(x)v=vx$, for all $x\in p P p$.
\item For any group $\mathcal G\subset \mathcal U(P)$ such that $\mathcal G''= P$ there is no sequence $(u_n)_n\subset \mathcal G$ satisfying $\|E_{ Q}(xu_ny)\|_2\rightarrow 0$, for all $x,y\in  M$.
\end{enumerate}
\end{thm} 
\vskip 0.02in
\noindent If one of the two equivalent conditions from Theorem \ref{corner} holds then we say that \emph{ a corner of $P$ embeds into $Q$ inside $M$}, and write $P\prec_{M}Q$.

%%%%%%%%%%%%%%%%%%%%%%%%%%%%%%%%%%%%%%%%%%%%%%%%%%%%%%%%%%%%%%%%%%%%%%%%%%%%%%%%%%%%%%%%%%%%%%%%%%%%%%%%%%%%%%%%%%%
%%%%%%%%%%%%%%%%%%%%%%%%%%%%%%%%%%%%%%%%%%%%%%%%%%%%%%%%%%%%%%%%%%%%%%%%%%%%%%%%%%%%%%%%%%%%%%%%%%%%%%%%%%%%%%%%%%%
%%%%%%%%%%%%%%%%%%%%%%                      QUASINORMALIZERS                    %%%%%%%%%%%%%%%%%%%%%%%%%%%%%%%%%%%
%%%%%%%%%%%%%%%%%%%%%%%%%%%%%%%%%%%%%%%%%%%%%%%%%%%%%%%%%%%%%%%%%%%%%%%%%%%%%%%%%%%%%%%%%%%%%%%%%%%%%%%%%%%%%%%%%%%
%%%%%%%%%%%%%%%%%%%%%%%%%%%%%%%%%%%%%%%%%%%%%%%%%%%%%%%%%%%%%%%%%%%%%%%%%%%%%%%%%%%%%%%%%%%%%%%%%%%%%%%%%%%%%%%%%%%

\subsection{Quasinormalizers of von Neumann subalgebras}

\noindent Given an inclusion $N \subseteq M$, the quasi-normalizer $\mathcal{QN}_M(N)$ is the $*$-subalgebra of $M$ consisting of all elements $x\in M$ such that there exist $x_1,x_2,...,x_k\in M$ satisfying  $N x\subseteq \sum_i x_i N$ and $x N \subseteq \sum_i N x_i$, \cite{Po99}. The von Neumann algebra $\mathcal{QN}_M (N)''$ is called the \emph{quasi-normalizing algebra of $N$ inside $M$}. This is an extension of normalization and it is precisely the von Neumann algebraic counterpart of the notion of commensurator in group theory. As usual, $\mathcal N_M(N)=\{u\in \mathcal U(M) \,:\, uNu^*=N \}$ denotes the \emph{normalizing group} and $ \mathcal N_M(N)''$ denotes the \emph{normalizing algebra of $N$ in $M$}. We obviously have  $N\subseteq N\vee N'\cap M \subseteq \mathcal N_M(N)''\subseteq \mathcal{QN}_M(N)''\subseteq M$. In general the quaisnormalizing algebra is (much) larger than the normalizer but there are natural instances when they coincide; e.g.\ when $N\subseteq M$ is a MASA it was shown in \cite{Po01} that $\mathcal{QN}_M(A)''=\mathcal N_M(A)''$.  Quasinormalizers play an important role in the classification of von Neumann algebras and over the last decade there have been a sustained effort towards computing these algebras in various situations \cite{Po01}.   
\vskip 0.05in

\noindent In this subsection we highlight some new computations of quasinormalizers of subalgebras in crossed products from \cite{CP18} that are essential to deriving our main results from Section 4. If $\G\ca^\sigma X$ is a free ergodic action and $M= L^\infty(X)\rtimes \G$ then $\mathcal {QN}_{M}(L(\G))''$ was computed in the following situations. When $\G$ is infinite abelian and $\sigma$ is weak mixing Nielsen observed that $L(\G)$ is a singular MASA in $M$ \cite{Ni}. Later Packer was able to show that the normalizer (and hence the quasinormalizer) depends only on the discrete spectrum of $\sigma$; more precisely one has $\mathcal {QN}_M(L(\G))''= L^\infty (X_c)\rtimes \G$, where $\G\ca X_c$ is the maximal compact factor of $\G\ca X$ \cite{Pa}. More recently Ioana obtained a far-reaching generalization  of Packer's result by showing that the same holds for every $\G$ and any ergodic action $\sigma$, \cite[Section 6]{Io08a}. In \cite{CP18} this analysis was completed at the entire level of the distal tower of $\G\ca X$ using iterated quasinormalizers.  
\vskip 0.05in

\noindent An action $\G\ca X$ is called \emph{distal} if it is the last element of an increasing finite or transfinite sequence $\G\ca X_\beta$ of factors $\beta\leq \alpha$, such that $\G\ca X_o$ is the trivial factor, each extension $\pi: X_{\beta+1}\ra X_\beta$ is maximal compact, and for every limit ordinal $\beta \leq\alpha$ the action $\G\ca X_\beta$ is the inverse limit of the preceding factors. The sequence $\{\G\ca X_\beta\}_{\beta\leq \alpha}$ of factors is also called the Furstenberg-Zimmer tower of $\G\ca X$. Furstenberg \cite{Fur77} and Zimmer \cite{Z76} independently obtained the following structure theorem
\begin{thm}Let $\G\ca X$ be any action. Then there exists an ordinal $\alpha$ and a unique distal tower $\{\G \ca X_\beta\}_{\beta \leq \alpha}$ such that the extension $\pi: X\ra X_\alpha$ is weak mixing.
%\begin{enumerate}
 %  \item $Y_\emptyset$ is a point;
 %  \item For every successor ordinal $\beta+1 \leq \alpha, \G\ca X_{\beta+1}$ is a compact extension of $\G\ca X_\beta$;
 %  \item For every limit ordinal $\beta \leq \alpha$, the action $\,\G\ca X_\beta $ is the inverse limit of $\G\ca X_\gamma$ for $\gamma < \beta$, in the se5nse that $L^2(Y_\beta)$ is the closure of $\bigcup_{\gamma < \beta} L^2(Y_\gamma)$;
  % \item $\G\ca X$ is a weak mixing extension of $\G\ca X_\alpha$.
  % \end{enumerate}
\end{thm}

\noindent In \cite{CP18} Peterson and the first author obtained a purely von Neumann algebraic way of describing Furstenberg-Zimmer distal tower of factors for an action, namely as towers of quasinormalizers.

 \begin{thm}\label{distaltowerCP}Let $\G\ca^\sigma X$ be an ergodic action and let $\{\G\ca X_\beta\}_{\beta\leq \alpha}$ be the corresponding Furstenberg-Zimmer tower. Let  $M=L^\infty(X)\rtimes \G$ and for all $\beta\leq \alpha$ let $M_\beta=L^\infty(X_\beta)\rtimes \G$ be the corresponding cross-products von Neumann algebras. Then the following hold: \begin{enumerate}
 \item for all $\beta\leq \beta'\leq \alpha$ we have the following inclusions of von Neumann algebras $L(\G)=M_o\subseteq M_\beta\subseteq M_{\beta'}\subseteq M_\alpha\subseteq M$; 
 \item  for all $\beta\leq \alpha$ we have $\mathcal {QN}_M(M_\beta)''= M_{\beta+1}$;
\item for every limit ordinal $\beta\leq \alpha$ we have $\overline{\cup_{\g<\beta} L^\infty(X_\g)}^{WOT}=L^\infty(X_\beta)$ and also $\overline {\cup_{\g<\beta}M_\g}^{WOT}= M_\beta$; \item There exists an infinite sequence $(\g_n)_n \subset \G$ such that for every $x,y\in L^\infty(X)\ominus L^\infty(X_\alpha)$ we have that $\lim_{n\ra \infty}\|E_{L^\infty(X_\alpha)}(x\sigma_{\g_n}(y))\|_2=0$.
\end{enumerate} \end{thm}

\subsection{Finite index inclusions of II$_1$ factors} A trace-preserving action $\G\ca A$ on a finite von Neumann algebra is called \emph{transitive} if $A$ is abelian and there exist finitely many minimal projections $\mathcal F\subset \mathcal P(A)$ such that span$\mathcal F= A$ and for every $p,q\in \mathcal F$ there is $\g\in \G$ such that $\sigma_\g(p)=q$.  Throughout the paper the set $\mathcal F $ will be denoted by $At(A)$ and will be called the atoms of $A$. In particular all atoms of $A$ have same trace, i.e. $dim(A)^{-1}$.
\begin{lem}\label{finindex1} Let $A$ be an abelian von Neumann algebra and let $\G \ca^\sigma A$ be a trace preserving action. Assume that the inclusion $L(\G)\subseteq A\rtimes \G$   
admits a finite Pimsner-Popa basis. Then $A$ is completely atomic. Moreover, if $\G \ca A$ is ergodic then $\G \ca A$ is transitive. \end{lem}

\begin{proof} By assumption there exist $m_1,...,m_k \in A\rtimes \G=M$, with $E_{L(\G)}(m_im_j^*)=\delta_{i,j} p_i$ where $p_i \in \mathcal P(M)$, such that for all  $x\in M$ we have $x= \sum_{i=1}^k E_{L(\G)}(x m^*_i) m_i$. Thus, for all $x\in M$ we have $\|x\|_2^2= \sum^k_{i=1} \|E_{L(\G)}(x m^*_i)\|^2_2$.  Approximating $m_i\in M$ using their Fourier decompositions and doing some basic calculations this  further implies the following: for every $\varepsilon>0$ one can find $a_j \in A$ with $1\leq j\leq l$ and $c>0$ so that for all $x\in (M)_1$ we have \begin{equation}\label{0.0.1}
\|x\|_2^2\leq \varepsilon+c\sum_{i=1}^l \|E_{L(\G)}(x a_i)\|^2_2.
\end{equation}   
Assume for the sake of contradiction that $A$ has a diffuse corner, i.e.\ there is $0\neq p \in A$ so that $Ap$ is diffuse. Hence one can find a sequence of unitaries $u_n\in \mathcal U(Ap)$ so that for all $x\in Ap$ we have $\tau (u_n x)\ra 0$, as $n\ra \infty$. Since $a_i\in A$ we have $E_{L(\G)}(u_n a_i) =\tau(u_n a_i)= \tau (u_n a_ip )$. Thus using \eqref{0.0.1} we get that  $\tau(p)=\|u_n\|_2^2\leq \varepsilon +c\sum_{i=1}^l \|E_{L(\G)}(u_n a_i)\|^2_2=\varepsilon +c\sum_{i=1}^l |\tau(u_n a_ip)|^2$ and since $\lim_{n\ra \infty}\sum_{i=1}^l |\tau(u_n a_ip)|^2=0$ we get that $\tau(p)\leq \varepsilon$. Letting $\varepsilon \searrow 0$ we get $p=0$, a contradiction. 
\vskip 0.05in
\noindent To see the moreover part let $0\neq q\in A$ be a minimal projection of maximal trace. Thus for all $\g\in \G$ either $q\sigma_\g (q)=0$ or $q=\sigma_\g (q)$. Thus the orbit $\mathcal F=\{ \sigma_\g(q)\,|\, \g\in \G\}$ is necessarily a finite set of (orthogonal) minimal projections of $A$. Let $t=\sum_{q\in \mathcal F} q$ and notice that $0\neq t\in A$ is a projection satisfying $\sigma_\g(t)=t$ for all $\g\in \G$. Since $\G \ca A$ is ergodic it follows that $t=1$. Since $A$ is completely atomic this entails that $A =span \mathcal F$. Thus $\G \ca A$ is transitive. \end{proof}

\begin{prop}\label{basicconstr1} Let $\G$, $\La$  be icc groups and let $\G \ca A$, $\La \ca B$ be transitive actions so that $A\rtimes \G$ and $B\rtimes \La$ are II$_1$ factors. Assume that $\theta: A\rtimes \G \ra B\rtimes \La$ is a $\ast$-isomorphism such that $\theta(L(\G))=L(\La)$. Then $\dim(A)=\dim(B)$ and for every $a\in At(A)$, $b\in At(b)$ there is a unitary $u\in L(\La)$ so that $\theta(L(Stab_\G(a))=u^* L(Stab_\La(b))u$. In addition, if there exists $a\in At(A)$ such that $Stab_\G(a)$ is normal in $\G$ then for every $b\in At(B)$  then $Stab_\La(b)$ is also normal in $\La$; moreover, $\G/Stab_\G(a)\cong \La/Stab_\La(b)$.  
\end{prop}

\begin{proof} To simplify the presentation, we assume that $A\rtimes \G= B\rtimes \La$ and $L(\G)=L(\La)$. Let $n=\dim(A)$ and fix $a\in At(A)$. Notice $\tau(a)=1/n$ and hence $E_{L(\G)}(a)=\tau(a)1= 1/n$. Also for each $x\in L(\G)$, using its Fourier decomposition,  we have $axa=\sum_{\g\in \G} \tau(xu_\g^{-1})a u_\g a= \sum_{\g\in \G} \tau(xu_\g^{-1})a \sigma_\g (a) u_\g= \sum_{\g\in Stab_\G(a)} \tau(xu_{\g^{-1}}) u_\g a= E_{L(Stab_\G(a))}(x)a$. Since clearly $\ast$-alg$\{a, L(\G)\}=A\rtimes \G$ then, altogether, the above relations show that $A\rtimes \G$ is the basic construction of the inclusion $L(Stab_\G(a))\subseteq L(\G)$ and also $[\G:Stab_\G(a)]=[A\rtimes \G :L(\G)]=n$. A similar statement holds for $L(\La)\subseteq B\rtimes \La$. Since by assumption $[A\rtimes \G:L(\G)]=[B\rtimes \La:L(\La)]$ it follows that $\dim (A)=\dim (B)=n$. To show the remaining part of the statement fix $b\in At(B)$. By the factoriality assumption, since $\tau(a)=\tau(b)=1/n$, there is a unitary $u\in A\rtimes \G$ so that \begin{equation}\label{0.2.0}b =uau^*.\end{equation} Since $a\in A\rtimes \G$ is the Jones projection for inclusion $L(Stab_\G(a))\subseteq L(\G)$, by pull-down lemma there exists $m\in L(\G)$ such that $b= uau^*=mam^*$. Thus one can check that 
$1/n=\tau(b)=E_{L(\La)}(b)= E_{L(\G)}(b)= E_{L(\G)}(mam^*)= mE_{L(\G)}(a)m^*=\tau(a)mm^*=(1/n)mm^*$. Hence $mm^*=1$ which implies that $m\in L(\La)$ is a unitary. Thus in equation \eqref{0.2.0} we can assume wlog that the unitary $u$ belongs to $L(\G)$. Hence using \eqref{0.2.0} we further have that $L(Stab_\La(b))=\{b\}'\cap L(\La)= \{uau^*\}'\cap L(\G)= \{uau^*\}'\cap  uL(\G)u^*= uL(Stab_\G(a))u^*$, as desired. Since $Stab_\G (a)$ is normal in $\G$ it follows from the above relation that $uu_\g u^* \in\mathcal N_{L(\La)}(L(Stab_\La(b)))$ for every $\g \in \G$. Since $\La$ is icc and $[\La:Stab_\La(b)]<\infty$ then $L(Stab_\La(b))\subseteq L(\La)$ is a irreducible inclusion of II$_1$ factors. Thus using \cite[Corollary 5.3]{SWW09} we have that for every $\g \in \G$ there exist a unitary $x\in L(Stab_\g(b))$ and $\la \in \La$ such that $uu_\g u^*= x v_\la$. In particular this implies that $Stab_\La(b)$ is normal in $\La$ and also $\G/Stab_{\G}(a)\cong \La/Stab_\La(b)$.   \end{proof}

\subsection{Mixing extensions}\label{prelim-mixing}

Let $B \subseteq A$ be an inclusion of von Neumann algebras and assume that $\G \ca^\sigma A$ is an action that leaves the subalgebra $B$ invariant. Throughout the paper we call such a system \emph{an extension} and we denote it by $\G\ca (B\subseteq A)$. When $A$ is endowed with a state $\phi$ preserved by $\sigma$ the extension is said to be $\phi$-preserving and will be denoted by $\G\ca (B\subset A, \phi)$. When $A$ is a finite von Neumann algebra and $\phi$ is a faithful normal trace then $\G\ca (B\subset A, \phi)$ is called a trace-preserving extension.

\begin{defn} A trace-preserving extension $\G\ca (B \subseteq A, \tau) $ is called mixing if for every $t,z\in A\ominus B$ we have $\lim_{\g\to \infty}\|E_{B}(t\sigma_\g (z))\|_2=0$. 
\end{defn}

\begin{lem}\label{mix1} Let $\G\ca (B \subseteq A, \tau) $ be a trace-preserving mixing extension. Then for every $t,z\in (A\rtimes \G) \ominus (B\rtimes \G)$ and every sequence $(x_n)_n\subset (L(\G))_1$ that converges to $0$ weakly, we have $\lim_{n\ra \infty}\|E_{B\rtimes \G} (t x_n z)\|_2=0$.
\end{lem}
\begin{proof} Fix $t,z\in (A\rtimes \G) \ominus (B\rtimes \G)$. Consider the Fourier decompositions $t= \sum_\g E_A(tu_{\g^{-1}})u_\g$ and $z= \sum_\g u_{\g^{-1}} E_A(u_{\g} z)$ and notice that $E_B(t u_{\g})= E_B( u_\g z)=0$ for all $\g\in \G$. 
Fix $\varepsilon>0$. Using these decompositions and basic $\|\cdot\|_2$-estimates one can find finite subsets $F,G\subset \G$ such that \begin{equation}\label{0.1.2'}
\|E_{B\rtimes \G}(t x_n z)\|_2\leq \frac{\varepsilon}{2} + \sum_{\delta\in F,\la\in G}  \| E_{B\rtimes \G} (E_A(t u_{\delta^{-1}}) u_\delta x_n u_{\la^{-1}}E_A(u_\la z))\|_2.
\end{equation}  

\noindent Also fix $a,b\in A$ and $x\in L(\G)$. Using the Fourier decomposition of $x\in L(\G)$ we see that $E_{B\rtimes \G} (axb)= \sum_\g \tau(xu_{\g^{-1}})E_{B\rtimes \G} (au_\g b)= \sum_\g \tau(xu_{\g^{-1}})E_{B} (a\sigma_\g (b)) u_\g$. Thus we have the formula \begin{equation}\label{0.1.1}\|E_{B\rtimes \G} (axb)\|_2^2=\sum_\g |\tau(xu_{\g^{-1}})|^2\|E_{B} (a\sigma_\g (b)) \|^2_2.\end{equation}
 Since $\G\ca (B \subseteq A, \tau) $ is mixing and $F,G$ are finite one can find a finite subset $H\subset \G$ so that $\|E_{B} (E_A(t u_{\delta^{-1}}) \sigma_\g (E_A(u_\la z)))) \|_2\leq \varepsilon/(\sqrt{8}|F||G|)$ for all $\g\in \G\setminus H$, $\delta\in F$ and $\la\in G$. Also since $x_n\ra 0$ weakly, and $F,G,H$ are finite there is an integer $n_0$ such that $|\tau(x_n u_{\g^{-1}})|\leq \varepsilon/(\sqrt{8 |H|}|G||F|\|t\|_\infty\|z\|_\infty)$ for all $\g\in G^{-1}H^{-1}F$ and $n\geq n_0$. Using these basic estimates in combination with formula \eqref{0.1.1} we see that for all $n\geq n_0$ we have \begin{equation*}\begin{split}
&\sum_{\delta\in F,\la\in G}  \| E_{B\rtimes \G} (E_A(t u_{\delta^{-1}}) u_\delta x_n u_{\la^{-1}}E_A(u_\la z))\|_2=\\
&= \sum_{\delta \in F,\la\in G} (\sum_{\g\in H} |\tau(x_n u_{\la^{-1}\g^{-1} \delta})|^2\|E_{B} (E_A(t u_{\delta^{-1}}) \sigma_\g (E_A(u_\la z)))) \|^2_2\\ &\quad +\sum_{\g\in \G\setminus H} |\tau(x_n u_{\la^{-1}\g^{-1} \delta})|^2\|E_{B} (E_A(t u_{\delta^{-1}}) \sigma_\g (E_A(u_\la z)))) \|^2_2)^{\frac{1}{2}}\\
&\leq \sum_{\delta \in F,\la\in G} \left( \sum_{\g\in H} \frac{\varepsilon^2}{8|F|^2 |G|^2|H|\|t\|^2_\infty \|z\|^2_\infty}\|E_{B} (t\sigma_\g (z)) \|^2_2+ \frac{\varepsilon^2}{8|F|^2|G|^2} \|x_n\|^2_2\right)^{\frac{1}{2}}\leq \left(\frac{\varepsilon^2}{8} +\frac{\varepsilon^2}{8}\right)^{\frac{1}{2}} =\frac{\varepsilon}{2}.
\end{split}
\end{equation*}
This combined with \eqref{0.1.2'} show that for every $\varepsilon>0$ there exits $n_0$ such that for all $n\geq n_0$ we have $\|E_{B}(t x_n z)\|_2\leq \varepsilon$, as desired.\end{proof}

\begin{thm}\label{mix2}  Let  $\G\ca (B \subseteq A, \tau) $ be a trace-preserving mixing extension. Also let $\La\ca^\alpha  (D \subseteq C, \tau) $ be a trace-preserving extension for which there exists a $\ast$-isomorphism  $\theta: A\rtimes \G \ra C\rtimes \La$ satisfying $\theta (B\rtimes \G)= D\rtimes \Lambda$ and $\theta(L(\G))=L(\La)$. Then $\La\ca (D \subseteq C, \tau) $ is a mixing extension. 
\end{thm}

\begin{proof} Suppressing $\theta$ from the notation we assume that $A\rtimes \G= C\rtimes \La$,  $B\rtimes \G=D\rtimes \La$ and $L(\G)=L(\Lambda)$. Fix $t,z \in C\ominus D$. We now show that for any infinite sequence $(\la_n)_n\subseteq \La$ we have that \begin{equation}\label{0.1.3}\lim_{n\ra \infty}\|E_D(t\alpha_{\la_n} (z))\|_2 =0.\end{equation}
Since $B\rtimes \G =D\rtimes \La$ we note that \begin{equation*}E_{B\rtimes \G} (z)=E_{D\rtimes \La}(z)= \sum_{\la\in \La}E_D(z v_{\la^{-1}})v_\la= \sum_{\la\in \La}E_D(E_C(z v_{\la^{-1}}))v_\la= \sum_{\la\in \La}E_D(z E_C( v_{\la^{-1}}))v_\la= E_D(z)=0.
\end{equation*} Similarly we have $E_{B\rtimes \G} (t)=0$. Since $t,z\in C$ and $D\rtimes \La =B\rtimes \G$ we see that \begin{equation}\label{0.1.2}\begin{split}\|E_D(t\alpha_{\la_n} (z))\|_2&=\|E_{D\rtimes \La} (t\alpha_{\la_n} (z))\|_2= \|E_{D\rtimes \La} (tv_{\la_n} z)v_{\la_n^{-1}}\|_2 \\&= \|E_{D\rtimes \La} (tv_{\la_n} z)\|_2= \|E_{B\rtimes \G} (tv_{\la_n} z)\|_2.	\end{split}\end{equation} 

\noindent Since $(\la_n)_n$ is infinite the sequence $(v_{\la_n})_n \subset L(\La)=L(\G)$ converges weakly to $0$. Thus applying Lemma \ref{mix1} we get $\lim_{n\ra\infty}\|E_{B\rtimes \G} (tv_{\la_n} z)\|_2=0$ and hence \eqref{0.1.3} follows from \eqref{0.1.2}. 
\end{proof}

\noindent For further use we recall the following technical variation of \cite[Theorem 3.1]{Po03}. The proof is essentially the same with the one presented in \cite{Po03} and will be left to the reader. 

\begin{thm}\label{mixingextn} Let Let  $\G\ca (B \subseteq A, \tau) $ be a trace-preserving mixing extension. Denote by $M= A\rtimes \G\supset B\rtimes \G=N$ the corresponding inclusion of crossed product von Neumann algebras. Then for every von Neumann subalgebra $C\subseteq N$ satisfying $C\nprec_N B$ we have $\mathcal {QN}_M(C)''\subseteq N$. 
\end{thm}

%%%%%%%%%%%%%%%%%%%%%%%%%%%%%%%%%%%%%%%%%%%%%%%%%%%%%%%%%%%%%%%%%%%%%%%%%%%%%%%%%%%%%%%%%%%%%%%%%%%%%%%%%%%%%%%%%%%
%%%%%%%%%%%%%%%%%%%%%%%%%%%%%%%%%%%%%%%%%%%%%%%%%%%%%%%%%%%%%%%%%%%%%%%%%%%%%%%%%%%%%%%%%%%%%%%%%%%%%%%%%%%%%%%%%%%
%%%%%%%%%%%%%%%%%%%%%%              INTERMEDIATE SUBALGEBRA PROPERTY            %%%%%%%%%%%%%%%%%%%%%%%%%%%%%%%%%%%
%%%%%%%%%%%%%%%%%%%%%%%%%%%%%%%%%%%%%%%%%%%%%%%%%%%%%%%%%%%%%%%%%%%%%%%%%%%%%%%%%%%%%%%%%%%%%%%%%%%%%%%%%%%%%%%%%%%
%%%%%%%%%%%%%%%%%%%%%%%%%%%%%%%%%%%%%%%%%%%%%%%%%%%%%%%%%%%%%%%%%%%%%%%%%%%%%%%%%%%%%%%%%%%%%%%%%%%%%%%%%%%%%%%%%%%

\section{Extensions satisfying the intermediate subalgebra property}

Let $\G \ca (P_0\subseteq P)$ be an extension of tracial von Neumann algebras and consider the corresponding inclusion $P_0\rtimes \G \subseteq P\rtimes \G$ of von Neumann algebras.  Suzuki discovered in \cite{Suzuki} that if $P_0$, $P$ are abelian and $\G\ca P_0$ is free then the extension $\G \ca (P_0\subseteq P)$ satisfies the \emph{intermediate subalgebra property}, i.e. every intermediate subalgebra $P_0\rtimes \G \subseteq N\subseteq P\rtimes \G$ arises as $N=Q\rtimes \G$ for some $\G$-invariant intermediate subalgebra $P_0\subseteq Q \subseteq P$. In this section we establish the intermediate subalgebra property for new classes of extensions (e.g. compact) for icc groups $\G$ (see Theorem \ref{cptint}). In many respects these results complement Suzuki's as they cover many examples of non free extensions, for instance when $P_0=\mathbb C 1$. As a consequence, for all free ergodic pmp actions on probability spaces $\G \ca X$ of icc groups $\G$, we are able to completely describe all intermediate subfactors $L(\G)\subseteq N\subseteq L^\infty(X)\rtimes \G$ with finite index $[N:L(\G)]<\infty$ (see Theorem \ref{intsubgen}). Our strategy also enables us to recover some well-known older results on intermediate subalgebras (see Corollary \ref{tensorint} and Theorems \ref{twistedGe}, \ref{galoiscorr}).     
\vskip 0.03in
\noindent We briefly introduce a few preliminaries. The first result describes the algebraic structure of fixed point subspaces associated with u.c.p.\ maps and it is essentially \cite[Lemma 3.4]{BJKW}. For reader's convenience we also include a short proof.
\begin{lem} \label{jl}
Let $M$ be a von Neumann algebra, and let $\varphi$ be a faithful, normal state on $M$. Let $\Psi: M \rightarrow M$ be a normal, u.c.p. map. 
Define $Har( \Psi )= \{ m \in M: \Psi(m)= m \} $ 
to be the fixed points of $\Psi$. 
If $\varphi \circ \Psi = \varphi$ then $Har(\Psi)$ is a von Neumann subalgebra of $M.$
\end{lem}

\begin{proof}
From the definition it is clear that $Har(\Psi)$ is closed under sum and taking adjoint. Also since $\Psi$ is normal, $Har(\Psi)$ is closed in the weak-operator topology. Thus, to finish the proof we only need to show that $Har(\Psi)$ is closed under product. Using the polarization identity, it suffices to show that whenever $x \in Har(\Psi)$ we have that $x^ \ast x \in Har(\Psi)$ as well. By Kadison-Schwarz inequality we have that $\Psi(x^ \ast x) \geq \Psi(x)^ \ast \Psi(x) =x ^ \ast x$, where the last equality follows because $x \in Har(\Psi)$; thus $\Psi(x^ \ast x) - x^ \ast x \geq 0 $. Since $\varphi \circ \Psi = \varphi$ we also have $\varphi(\Psi(x^ \ast x) - x^ \ast x)=0$. Since $\varphi$ is faithful, we get that $\Psi(x^ \ast x)= x^ \ast x$, thereby proving that $Har(\Psi)$ is an algebra.
\end{proof}

\begin{thm} \label{usefulresult}
Let $\Gamma \ca (P, \tau)$ be a trace preserving action on a finite von Neumann algebra $P$ and consider the corresponding crossed product von Neumann algebra $P\rtimes \G$.  Let $P_0 \subseteq P$ be a $\G$-invariant subalgebra.
Assume that $P_0 \rtimes \G \subseteq N \subseteq P \rtimes \Gamma $ is an intermediate von Neumann subalgebra. Then there is a $\G$-invariant subalgebra $P_0 \subseteq Q\subseteq P$  so that $N =Q \rtimes \Gamma$ if and only if $E_N(P)\subseteq P$. 

\end{thm}

\begin{proof} Denote by $M =P\rtimes \G$ and let $E_{P}: M \ra P$ and $E_N: M \ra N$ be the canonical conditional expectations onto $P$ and $N$, respectively. To see the direct implication, fix $a\in P$. Since $N= Q\rtimes \G$ and $L(\G) \subseteq N$ we have \begin{equation*}
\begin{split}
E_N(a)&= \sum_\g E_{Q}(E_N(a) u_{\g^{-1}})u_\g=  \sum_\g E_{Q}(E_N(a u_{\g^{-1}}))u_\g \\&= \sum_\g E_{Q}(a u_{\g^{-1}})u_\g=\sum_\g E_{Q}(E_{P}(a u_{\g^{-1}}))u_\g \\&= \sum_\g E_{Q}( a E_{P}(u_{\g^{-1}}))u_\g= E_{Q}(a)\in P.
\end{split} 
\end{equation*}

\noindent Next we show the reverse implication. Let $e_P: L^2(M)\ra L^2(P)$ and $e_N: L^2(M)\ra L^2(N)$ be the canonical orthogonal projections. Since $E_N(P)\subseteq P$ then $E_{P}( E_N(a))=  E_N( a)$ for all $a\in P$. Therefore $E_{P}\circ E_N \circ E_{P}= E_N \circ E_{P}$ and hence $e_{P} e_N e_{P}= e_N e_{P}$.  Taking adjoints we obtain $e_N e_{P}=e_{P}e_N$ and since $(e_{P} e_N e_{P})^n$ converges to $e_{N} \wedge e_{ P}$ in the strong-operator topology, as $n$ tends to infty, we conclude that $e_{P}e_N =e_N e_{P}=e_{N}\wedge e_{P}$. This also entails that $e_N\wedge e_{P}= e_{N \cap P}$ and thus \begin{equation}\label{2} E_N\circ E_{P}=E_{P}\circ E_N= E_{N\cap P}. 
\end{equation}

\noindent Alternatively, one can show \eqref{2} just by using Lemma \ref{jl}.  Indeed since $E_N(P) \subseteq P$ then from assumptions ${E_N}_{|_P}: P \rightarrow P$ is a u.c.p. map which preserves $\tau$, a normal, faithful, tracial state. Letting $\Psi= {E_N}_{|_P}$ we can easily see that $N\cap P\subseteq Har(\Psi) \subseteq E_N(P)$. Since we canonically have $E_N(P)\subseteq N\cap P$ we conclude that $Har(\Psi)= E_N(P)=P\cap N$. The last equality gives \eqref{2}. 
\vskip 0.06in 
\noindent Notice that from assumptions $Q:=N\cap P\subseteq P$ is a $\G$-invariant von Neumann subalgebra of $P$ containing $P_0$. So to finish the proof of our implication we only need to show that $N = Q\rtimes\G$. Since $Q\rtimes\G\subseteq N$ canonically, we will only argue for the reverse inclusion. To see this fix $x\in N$ and consider its Fourier decomposition (in $M$) $x=\sum_\g x_\g u_\g$ where $x_\g\in P$. Since $L(\G)\subseteq N$ we have  $\sum_\g x_\g u_\g =x=E_N(x)= E_N(\sum_\g x_\g u_\g)= \sum_\g E_N(x_\g) u_\g$. By \eqref{2} we have  $E_N (x_\g)= E_{Q}(x_\g)\in Q$ and hence $x_\g= E_{Q}(x_\g)\in Q$ for all $\g \in \G$. Thus  $x= \sum_\g E_{Q}(x_\g)u_\g \in Q\rtimes \G$, as desired.
\end{proof}
\noindent The conditional expectation property presented in the previous theorem can be used effectively to describe all the intermediate subalgebras for many inclusions arising from canonical constructions in von Neumann algebras. In the remaining part of the section we highlight several situations when this is indeed the case. For instance it provides a very fast approach to Ge's well known tensor-splitting theorem \cite[Theorem 3.1]{Ge} for finite von Neumann algebras.

\begin{cor}\label{tensorint}{\rm ([Ge, Theorem 3.1])} Let $P_1$ be a factor and let $P_2$, $N$ be  von Neumann algebras such that $P_1 \otimes 1 \subseteq N \subseteq P_1 \bar \otimes P_2$. Assume there exist faithful normal states $\varphi_1$ on $P_1$ and $\varphi_2$ on $P_2$, and a faithful, normal conditional expectation $E_N: P_1 \bar\otimes P_2 \rightarrow N$ preserving $\varphi:= \varphi_1 \otimes \varphi_2$.
Then $N = P_1 \bar \otimes Q$ for some (von Neumann) subalgebra $Q \subseteq P_2$.
\end{cor}

\begin{proof}
We first claim that $E_N(1 \otimes P_2) \subseteq 1 \otimes P_2$. To see this fix $p_2 \in P_2$ and $p_1 \in P_1$. Since $N \supset P_1 \otimes 1$ we have $(p_1 \otimes 1)E_N(1 \otimes p_2)= E_N(p_1 \otimes p_2)= E_N((1 \otimes p_2)(p_1 \otimes 1))= E_N(1 \otimes p_2)(p_1 \otimes 1).$ This implies that $E_N(1 \otimes p_2) \in (P_1 \otimes 1)' \cap (P_1 \bar \otimes P_2)= 1 \otimes P_2$, thereby proving the claim. So we have that $E_N: 1 \otimes P_2 \rightarrow 1 \otimes P_2$ is a u.c.p. map, preserving $\varphi$,  a faithful, normal state. So by Lemma \ref{jl}, $E_N(1 \otimes P_2)$ is a subalgebra of $1 \otimes P_2$, which we can identify as a von Neumann subalgebra $Q \subseteq P_2$. Under this identification we have that $1 \otimes Q= E_N(1 \otimes P_2)= N \cap (1 \otimes P_2).$ Hence $P_1 \tp Q \subseteq N$.
\vskip 0.03in
\noindent To show the reverse containment, we first claim that $(P_1 \otimes_{\rm alg}P_2) \cap N$ is WOT-dense in $N$. Let $n \in N$. By Kaplansky's density theorem, we can find a bounded net $(x_{\la})\subset P_1 \otimes_{\rm alg}P_2$ such that $x_{\la} \rightarrow n$ in WOT. Since $E_N$ is normal, we get that $E_N(x_{\la}) \rightarrow n$. If $x_{\la}= \sum_i p_i \otimes q_i$, with $p_i \in P_1$ and $q_i \in P_2$ then $E_N(x_{\la})=E_N(\sum_i p_i \otimes q_i)= \sum_i(p_i \otimes 1)E_N((1 \otimes q_i)) \in (P_1 \otimes_{\rm alg}P_2) \cap N$, thereby establishing the claim.
Now fix $n \in (P_1 \otimes_{\rm alg}P_2) \cap N.$ Then, there exist $p_1,..., p_k \in P_1$ and $q_1,..., q_k \in P_2$ so that $n = \sum_i p_i \otimes q_i.$ Now, since $n \in N$ and $P_1\otimes 1\subseteq N$ we have 
\begin{equation} \label{teneq}
n = E_N(n)= E_N(\sum\limits_{i=1}^k p_i \otimes q_i) = \sum\limits_{i=1}^k E_N(p_i \otimes q_i)= \sum\limits_{i=1}^k (p_i \otimes 1) E_N(1 \otimes q_i). 
\end{equation}
Since $E_N(1 \otimes q_i) \in Q$ then \eqref{teneq} implies that  $(P_1 \otimes_{\rm alg}P_2) \cap N \subseteq P_1 \otimes_{\rm alg} Q$. As $(P_1 \otimes_{\rm alg}P_2) \cap N $ is WOT-dense in $N$ we get $N= P_1 \bar \otimes Q.$  
\end{proof}

\noindent We also record a twisted version of the above theorem. 

\begin{thm} \label{twistedGe}
Let $P$ be a $\rm II_1$ factor and let $Q$ be a finite separable von Neumann algebra, equipped with a trace preserving action of $\G$. Assume that $\G \ca^{\sigma} P$ is an outer action. Then for any intermediate von Neumann subalgebra $P \rtimes \G \subseteq N \subseteq (P \bar \otimes Q) \rtimes \G$ there is a von Neumann subalgebra $Q_0\subseteq Q$ such that $N=(P  \bar \otimes Q_0) \rtimes \G.$	
\end{thm}

\begin{proof} Using Theorem \ref{usefulresult} we only need to show that $E_N(Q) \subseteq Q$. Naturally, we have that $E_N(Q)\subseteq P' \cap (P\otimes Q)\rtimes \G$.  We shall now briefly argue that $P' \cap (P \bar \otimes Q) \rtimes \G \subseteq Q,$ which will prove our claim. To see this fix $\sum_{\g}a_{\g}u_{\g} \in P' \cap (P \bar \otimes Q) \rtimes \G,$ where $a_{\g} \in P \bar \otimes Q$. Thus for every $\g \in \G$ and $p\in P$ we have that $pa_{\g}=a_{\g}\sigma_{\g}(p)$. Fix $e \neq \g  \in \G$. Let $a_{\g}=\sum_ip_i \otimes q_i$, with $p_i \in P$ and $q_i \in Q$. We may assume that $q_i$ are orthogonal with respect to $\tau_Q$ (by using the Gram-Schimdt process, and using the separability of $Q$). Thus we have $\sum_i (pp_i)\otimes q_i=p(\sum_i p_i \otimes q_i)= (\sum_i p_i \otimes q_i)\sigma_\g(p)= \sum_i (p_i \sigma_\g(p))\otimes q_i $. As $q_i$'s are orthogonal we further get $pp_i = p_i \sigma_{\g}(p)$ for all $i$ and $p \in P$. Since $\G\ca P$ is outer, this implies $p_i=0$ for all $i$ and hence $a_\g=0$.  Thus, $P' \cap (P \bar\otimes Q) \rtimes \G \subseteq P ' \cap (P \bar \otimes Q)= Q.$
	Hence $N = Q_0 \rtimes \G$ where $Q_0= E_N(Q).$
\end{proof}

\noindent If $P$ is a $\rm II_1$ factor then an action $\G\ca P$ is called \textit{centrally free} if the induced action $\G\ca P' \cap P^{\omega}$ is properly outer (see \cite[Definition 4.3]{Suzuki}). Theorem \ref{twistedGe} was first obtained by Y. Suzuki under the assumption that the $\G\ca P$ is centrally free, \cite[Example 4.14]{Suzuki}. In general the centrally freeness assumption introduces certain limitations. For instance, if $P=L(\mathbb F_2)$ then $P' \cap P^{\omega}=\mc$ and hence no nontrivial group admits a centrally free action on $P$. However, when $P$ is the hyperfinite $\rm II_1$ factor, then requiring the $\G\ca P$ to be outer is the same as requiring the $\G\ca P$ to be centrally free. This surprising result is a consequence of Ocneanu's central freedom lemma (\cite[Lemma 15.25]{EK98}). The reader may also consult \cite{CD18} for another recent application of the central freedom lemma.

\begin{thm}
	Let $\mr$ denote the hyperfinite type $\rm II_1$ factor and let $\G$ be a discrete group acting on $\mr$. Then $\G \ca^{\sigma} \mr$ is outer if and only if $\G \ca^{\sigma} \mr$ is centrally free.
\end{thm}

\begin{proof}
	Let $\G \ca^{\sigma} \mr$ be an outer action. Let $a \in \mr' \cap \mr^{\omega}$ and $\g \in \G$ be such that $\sigma_{g^{-1}}(x)a=ax$ for all $x \in \mr' \cap \mr^{\omega}$. This clearly implies that $u_{\g}a \in (\mr' \cap \mr^{\omega})' \cap (\mr \rtimes \G)^{\omega}$. Now, by Ocneanu's central freedom lemma we get that $(\mr' \cap \mr^{\omega})' \cap (\mr \rtimes \G)^{\omega}=\mr \vee (\mr' \cap \mr \rtimes \G)^{\omega}=\mr$ (where the last equality holds because $\G \ca \mr$ is outer). Thus $u_\g  \in \mr $ which implies that $\g=e$. Hence $\G \ca^{\sigma} \mr' \cap \mr^{\omega}$ is outer.
	\vskip 0.05in
	\noindent Conversely, assume that $\G \ca^{\sigma} \mr' \cap \mr^{\omega}$ is outer. We will show that $\mr' \cap \mr \rtimes \G = \mc$, which shall establish that $\G \ca^{\sigma} \mr$ is outer. Let $x \in \mr ' \cap \mr \rtimes \G$, and consider its Fourier decomposition $x=\sum_{\g}x_\g u_\g$, where $x_\g \in R$. Now $x \in \mr ' \cap \mr \rtimes \G $ implies that $ x_\g u_\g \in \mr ' \cap \mr \rtimes \G$ for all $\g \in \G$. Hence $x_\g u_\g \in \mr \vee (\mr' \cap \mr \rtimes \G)^{\omega} = (\mr' \cap \mr^{\omega})' \cap (\mr \rtimes \G)^{\omega} $ (where the last equality follows from Ocneanu's central freedom lemma). Thus we get $x_\g u_\g x=xx_\g u_\g$ for all $x \in \mr' \cap \mr^{\omega}$ which gives that $x_\g\sigma_\g(x)= xx_\g $, implying $ \sigma_\g(x)x_\g x_\g^{\ast}= x_\g x_\g^{\ast}x$, which implies $E_{\mr' \cap \mr^{\omega}}(x_\g x_\g^{\ast})x= xE_{\mr' \cap \mr^{\omega}}(x_\g x_\g^{\ast})$, for all $x \in \mr' \cap \mr^{\omega}$.
	
	\noindent Since $\G \ca \mr' \cap \mr^{\omega}$ is outer, we get that $E_{\mr' \cap \mr^{\omega}}(x_\g x_\g^{\ast})=0$ for all $\g \neq e$. Since $E_{\mr' \cap \mr^{\omega}}$ is faithful, this further implies that $x_\g =0$ for all $\g \neq e$. This implies that $x \in \mr' \cap \mr = \mc$, thereby establishing that $\mr ' \cap \mr \rtimes \G =\mc$, which implies that $\G \ca^{\sigma} \mr$ is outer. 
\end{proof}

\noindent Theorem \ref{twistedGe}  leads to new examples of subalgebras in $(P \bar\otimes Q) \rtimes \G$ that are amenable relative to $P \rtimes \G$, \cite[Definition 2.2]{OP07}. Note that for von Neumann algebra inclusions $N \subseteq M$, the existence of a maximal amenable subalgebra $P$ in $M$ relative to $N$ follows from \cite[Lemma 2.7]{DHI16}. We remark that very similar methods were used in \cite[Theorem 3.4]{JS} to provide examples of maximal Haagerup subalgebras arising from \textit{extremely rigid actions} of an icc group.
\begin{cor}
	Let $P$ be a type $\rm II_1$ factor, let $Q$ be a finite von Neumann algebra, and let $\G$ be an amenable group acting outerly on $P, Q$ (the actions are assumed to be trace preserving). Let $Q_0\subseteq Q$ be a maximal amenable subalgebra. Then $(P \bar\otimes  Q_0) \rtimes \G$ is a maximal amenable subalgebra in $(P \bar\otimes Q) \rtimes \G$ relative to $P \rtimes \G$. In particular, if $\mr$ is the hyperfinite $\rm II_1$ factor, and $\G \ca \mr$ is an outer trace preserving action, then $(\mr \bar\otimes Q_0) \rtimes \G$ is maximal amenable in $(\mr \bar\otimes Q) \rtimes \G.$
\end{cor}

\begin{proof} Let $(P \tp Q_0) \rtimes \G \subseteq N \subseteq (P \tp Q) \rtimes \G$. Then by Theorem \ref{twistedGe} $N=(P \tp Q_1) \rtimes \G$, with $Q_0 \subseteq Q_1 \subseteq Q$. If $N$ is amenable relative to $P \rtimes \G$, then $Q_1$ is amenable. By maximal amenability of $Q_0$ we obtain that $Q_1=Q_0$ thereby establishing the result. 
\end{proof}

\noindent The next theorem re-establishes a well known Galois correspondence for group actions.

\begin{thm} \label{galoiscorr} Let $\G$ be a group,  let $\La\lhd \G$ be a normal subgroup, and let $(P, \tau)$ be a tracial von Neumann algebra. Assume that $\G$ acts on $P$ via trace preserving automorphisms such that $(P \rtimes \La)' \cap (P \rtimes \G)=\mc$. Then for any intermediate subfactor $P \rtimes \La \subseteq N \subseteq P \rtimes \G$ there exists an intermediate subgroup $\La \leqslant K \leqslant \G$ such that $N =P \rtimes K$.
\end{thm}

\begin{proof}
	Let $K =\{\g \in \G: u_{\g} \in N \}$. Clearly, $K$ is a group satisfying  $\La \leqslant K \leqslant \G$. Also $P \subseteq P \rtimes \La \subseteq N$ and hence $ P \rtimes K  \subseteq N \subseteq P \rtimes \G$. Next we show that $N \subseteq P \rtimes K$.
\vskip 0.03in	
\noindent First we claim for every $\g\in \G$ there is $c_\g\in \mc$ so that $E_N(u_\g)=c_\g u_\g$. Fix $\g \in \G$ and let $\psi(x)=u_{\g}xu_{\g}^\ast$,  for all $x \in L(\G)$. Since $\La$ is normal in $\G$, $\psi$ restricts to an automorphism of $P \rtimes \La$. Thus for all $x \in P \rtimes \La$ we have  $\psi(x)u_\g =u_\g x$ and hence $\psi(x)E_N(u_\g)= E_N(u_\g)x$. This implies that  $E_N(u_\g)^\ast \psi(x)= x E_N(u_\g)^\ast$  and hence $E_N(u_\g)E_N(u_\g)^\ast \in (P \rtimes \La)' \cap (P \rtimes \G) = \mc$. Let $d = E_N(u_\g)E_N(u_\g)^\ast$. Note that $ 0 \leq d \leq 1$. If $d \neq 0$, we get that $(d^{-1/2}E_N(u_\g))(d^{-1/2}E_N(u_\g))^\ast=1$, implying that $d^{-1/2}E_N(u_\g)\in \mathcal U(N)$. Next consider $u = u_\g^\ast E_N(d^{-1/2}u_\g)\in \mathcal U(M)$. For every $x \in P \rtimes \La$ we can check that
\begin{equation*} \label{autan2}
\begin{split}
uxu^\ast &= d^{-1}u_\g^\ast E_N(u_\g)xE_N(u_\g)^\ast u_\g = d^{-1} u_\g^\ast E_N(u_\g)E_N(u_\g)^\ast \psi(x) u_\g = u_\g^\ast \psi(x) u_\g=x. 
\end{split}
\end{equation*}
Hence  $d^{-1/2} u_\g^ \ast E_N(u_\g)=u \in (P \rtimes \La)' \cap (P \rtimes \G) = \mc$. Thus   $E_N(u_\g)=c_\g u_\g$ for some $c_\g\in \mc$.
\vskip 0.03in
\noindent The claim shows that for any $\g \in\G$, either $E_N(u_\g) =0$ or $u_\g \in N$. Finally, if $N\ni  n = \sum_{\g \in \G} n_\g u_\g$ is its Fourier decomposition in $P \rtimes \G$, then applying $E_N$, we see that $n= \sum_{\g \in \G} n_\g E_N(u_\g) \in P \rtimes K$, as desired.\end{proof}

\noindent Below we highlight a few special cases of the above theorem, which are well known in the literature.
\begin{cor}

\begin{enumerate}
		
\item (\cite{Ch},\cite[Theorem 3.13]{ILP98}) Let $M$ be a $\rm II_1$ factor, and let $\G$ be a discrete group with an outer action on $M$. Let $N$ be an intermediate subalgebra, i.e. $M \subseteq N \subseteq M \rtimes \G$. Then there exists a subgroup $K$ of $\G$ such that $N=M \rtimes K$. 
%\item Let $M$ be a $\rm II_1$ factor, and let $\G$ be a discrete group with an outer action on $M$. Let $\La$ be a normal subgroup of $\G$, with $(M^{\La})' \cap M^{\G} = \mc$. If $M^{\La} \subseteq N \subset M^{\G}$, then there exists a subgroup $K$ of $\G$ such that $N=M^{K}$.
		\item Let $\G$ be an icc group, and let $\La\lhd \G$ be a normal subgroup such that $L(\Lambda)' \cap L(\G) = \mc$. Then for any intermediate subfactor $L(\Lambda) \subseteq N \subseteq L(\G)$ there exists an intermediate subgroup $\La \leqslant K \leqslant \G$ such that $N =L(K)$.
	\end{enumerate}
\end{cor}

\begin{proof}
	Since $\G \ca M$ is outer, $M' \cap M \rtimes \G =\mc$. Taking $\La=\{e\}$, and appealing to Theorem \ref{galoiscorr} yields the first statement. 
%	Since $M^{\G} \subseteq M \subseteq M \rtimes \G$ is the basic construction, we get the second statement.
	Taking $P =\mc$ in Theorem \ref{galoiscorr} yields the second statement.
\end{proof}

\vskip 0.03in

\noindent In the remaining part of the section we show that the strategy presented in Theorem \ref{usefulresult} can be successfully used to classify all intermediate subalgebras for inclusion of von Neumann algebras arising from compact extensions. This covers a new situation which complements the case of free extensions discovered in \cite[Main Theorem]{Suzuki}. To be able to properly introduce our result we first recall the following notion of compact extension of actions on von Neumann algebras: 
\begin{defn}\label{compact}
Let $\G\ca (P_0 \subseteq P)$ be an extension of tracial von Neumann algebras. One says that $\G \ca (P_0\subseteq P)$ is a \emph{compact extension} if there exists $\mathcal F\subseteq P$ satisfying the following properties:
\begin{enumerate}\item $\overline{{\rm span} \mathcal F}^{\|\cdot\|_2}=L^2(P)$; 
\item for every $f\in \mathcal F$ and $\varepsilon>0$ there exist $\xi_1,\xi_2,...,\xi_n\in L^2(P)$ such that for every $\g\in \G$ one can find $\kappa_i(\g)\in P_0$, with $i=\overline{1,n}$ satisfying $\sup_{1\leq i\leq n,  \g\in \G}\|\kappa_i(\g)\|_\infty<\infty$ and  \begin{equation*}
\|\sigma_\g(f)- \sum^n_{i=1} \kappa_i(\g) \xi_i\|_2\leq\varepsilon.
\end{equation*}
\end{enumerate}   
 When $P_0=\mathbb C 1$ we simply say that the action $\G \ca P$ is compact.
\end{defn}

\noindent{\bf Examples.} Assume that $\G \ca X$ is an ergodic pmp action on a probability space $X$ and let $\G \ca X_0$ be a factor such that the extension $\pi: X \ra X_0$ is compact in the usual sense \cite{Fur77,Z76}. Then it is a routine exercise to show that the corresponding von Neumann algebraic extension $\G \ca (L^\infty (X_0)\subseteq L^\infty(X))$ automatically satisfies the definition above. In particular whenever $\G \ca X$ is an ergodic compact pmp action then $\G\ca L^\infty(X)$ is compact in the above sense.

\vskip 0.05in

\noindent With this definition at hand we can now introduce the main result of this section.

\begin{thm} \label{cptint}
Let $\G$ be an icc group and let $\G \ca (P_0\subseteq P)$ be a compact extension of tracial von Neumann algebras as in Definition \ref{compact}. Let $P_0\rtimes \G \subseteq P \rtimes \G$ be the corresponding inclusion of crossed product von Neumann algebras. Then for any intermediate von Neumann subalgebra $P_0\rtimes \Gamma \subseteq N \subseteq P \rtimes \Gamma $ there exists an intermediate von Neumann subalgebra $P_0\subseteq Q \subseteq P$ such that $N= Q \rtimes \G.$
\end{thm}

\begin{proof} Let $ M= P \rtimes \Gamma$. Denote by $E_N: M \rightarrow N$ the canonical trace preserving conditional expectation and note that it extends to a map from $L^2(M) \rightarrow L^2(M)$ by $E_N(\hat{m})= \widehat{E_N(m)}$. Similarly, let $E: M \rightarrow P$ be the trace preserving conditional expectation. E also extends to a map $E: L^2(M) \rightarrow L^2(M)$. For every $\xi\in L^2(M)$ let $\tilde \xi = E_N(\xi)-E\circ E_N(\xi)$. With these notations at hand we prove the following 
\begin{claim}\label{2.5.10} for every $\xi\in \mathcal F$ and every $\varepsilon>0$  there exists a finite set $K\subset \G\setminus \{e\}$ and $\eta_1,\eta_2,...,\eta_n\in span P K$ such that for every $\g\in \G$ there exist $\kappa_i(\g)\in P_0$ with $\sup_{\g\in \G }\|\kappa_i(\g)\|_\infty<\infty$ such that \begin{equation}
\|\sigma_\g(\tilde \xi)-\sum_i \kappa_i(\g)\eta_i\|_2\leq \varepsilon. 
\end{equation}     
\end{claim}  
 
\noindent \emph{Proof of {\bf Claim} \ref{2.5.10}.} First notice that since $L(\G)\subseteq N$ and $P$ is $\G$-invariant then for all $\xi \in L^2(M)$ and $\g\in \G$  we have 
 \begin{equation}\label{2.5.2}
 \begin{split}
   E_N(u_{\g} \xi u_{\g}^{\ast})= u_{\g}E_N( \xi)u_{\g}^{\ast}, \text{ and }
  E(u_{\g} \xi u_{\g}^{\ast})= u_{\g}E(\xi)u_{\g}^{\ast}.
\end{split}
\end{equation}
 
\noindent Fix $\xi\in \mathcal F$ and $\varepsilon>0$.  Since $\Gamma \curvearrowright^\sigma P_0\subseteq P$ is a compact extension there is a finite set $\xi_1,\xi_2,...,\xi_n\in L^2(P)$ such that for every $\g\in \G$ there exist $\kappa_i(\g)\in P_0$ with $\sup_{\g\in \G }\|\kappa_i(\g)\|_\infty<\infty$ so that \begin{equation}\label{2.5.11}
\|\sigma_\g( \xi)-\sum_i \kappa_i(\g)\xi_i\|_2\leq\frac{\varepsilon}{3}. 
\end{equation}

\noindent Using \eqref{2.5.2} in combination with \eqref{2.5.11} and the basic inequalities $\|E_N(m)\|_2, \|E\circ E_N(m)\|_2\leq \|m\|_2$, for all $m\in L^2(M)$ we get that 

\begin{equation*}
\begin{split}\|\sigma_\g(E_N(\xi))-\sum_i \kappa_i(\g)E_N(\xi_i)\|_2   
\leq \frac{\varepsilon}{3}, \text{ and } 
\|\sigma_\g(E\circ E_N(\xi))-\sum_i \kappa_i(\g)E\circ E_N(\xi_i)\|_2   
\leq \frac{\varepsilon}{3}
\end{split}
\end{equation*} 

\noindent Subtracting these relations and using the triangle inequality we conclude that 
\begin{equation}\label{2.5.12}
\|\sigma_\g(\tilde \xi))-\sum_i \kappa_i(\g) \tilde \xi_i)\|_2\leq \frac{2\varepsilon}{3}. 
\end{equation} 

\noindent Approximating the $\xi_i$'s one can find a finite set $F\subset \G\setminus\{e\}$ so that $\| \tilde \xi_i -\eta_i\|_2\leq \varepsilon/(3n \sup_{\g\in \G}\|\kappa_i(\g)\|_\infty)$ for all $1\leq i\leq n$. Thus $\|\sum_i \kappa_i(\g)\tilde \xi_i- \kappa_i(\g)\eta_i \|_2\leq \varepsilon/3$ and combining it with \eqref{2.5.12} we get the desired conclusion. $\hfill\blacksquare$
\vskip 0.04in
\noindent Next we prove the following  
 \begin{claim}\label{2.5.6} For every $\xi\in \mathcal F$ we have $\tilde \xi =0$.  \end{claim}

\noindent \emph{Proof of the Claim \ref{2.5.6}}. Fix $\varepsilon > 0$ and $\xi\in \mathcal F$. Approximating $\tilde \xi$ there exists a finite set $K\subseteq \G\setminus \{e\}$ and $r\in {\rm span} P K$ such that 
\begin{equation}\label{2.5.15}\|\tilde \xi-r\|_2 \leq \varepsilon.
\end{equation}

\noindent Also by Claim \ref{2.5.10} there exists a finite set $G\subset \G\setminus \{e\}$ and $\eta_1,\eta_2,...,\eta_n\in \rm span P G$ such that for every $\g\in \G$ there exist $\kappa_i(\g)\in P_0$ with $\sup_{\g\in \G }\|\kappa_i(\g)\|_\infty<\infty$ such that \begin{equation}\label{2.5.14}
\|\sigma_\g(\tilde \xi)-\sum_i \kappa_i(\g)\eta_i\|_2\leq \varepsilon. 
\end{equation}

 \noindent Since $\G$ is icc and $G, K\subset \G\setminus\{e\}$ are finite by  \cite[Proposition 3.4]{CSU13} there exists $\la \in \G$ such that $\la K \la^{-1} \cap G =\emptyset $; in particular, we have  \begin{equation}\label{2.5.8}\langle u_{\la} r u_{\la^{-1}}, \sum_i \kappa_i(\la) \eta_i   \rangle=0.
 \end{equation}  Using \eqref{2.5.15} in combination with Cauchy-Schwarz inequality, \eqref{2.5.14}, and \eqref{2.5.8} we see that
\begin{align*}
\|\tilde\xi\|_2^2 &= \|u_{\la} \tilde\xi u_{\la}^{\ast}\|_2^2 = |\langle u_{\la} \tilde \xi u_{\la^{-1}},  u_{\la} \tilde\xi u_{\la^{-1}} \rangle| \\
 &= \varepsilon \|\tilde \xi\|_2 + \langle u_{\la} r u_{\la^{-1}}, u_\la\tilde \xi u_{\la^{-1}}  \rangle | \leq \varepsilon\|\tilde \xi\|_2  + \varepsilon \|r\|_2+ |\langle u_{\la} r u_{\la^{-1}}, \sum_i \kappa_i(\la) \eta_i   \rangle|\\\
& \leq \varepsilon\|\tilde \xi\|_2  + \varepsilon (\|\tilde \xi\|_2+\varepsilon).  
 \end{align*} 
Letting $\varepsilon \searrow 0$ we get $\tilde\xi=0$, as desired. $\hfill\blacksquare$
\vskip 0.04in 
 \noindent Claim \ref{2.5.6} implies that $E_N(\xi)=E \circ E_N(\xi)$ for all $\xi\in \mathcal F$. Since ${\rm span} \mathcal F$ is dense in $L^2(P)$, these two maps agree on $L^2(P) \supseteq P$. Appealing to Theorem \ref{usefulresult} we conclude that $N =Q \rtimes \G$, for some subalgebra $P_0\subseteq Q \subseteq P$. \end{proof}

\noindent{\bf Remarks.} After the first draft of the paper appeared on the ArXiv, we were kindly informed by Y. Jiang and A. Skalski that they had subsequently obtained the same characterization of intermediate subfactors $N$ satisfying $L(\G) \subseteq N \subseteq  L^{\infty}(X) \rtimes \G$, with $\G \ca X$ a profinite action, in an independent manner (see \cite[Corollary 3.11]{JS}).

\begin{cor}\label{twisted2} Let $\G$ be an icc group and let $\G\ca A$ and $\G\ca B$ be trace preserving actions with $\G \ca B$  compact, and $A$ is a $\rm II_1$ factor. Consider the diagonal action $\G\ca A\bar\otimes B$ and let $(A\bar\otimes B) \rtimes \G$ be the corresponding crossed product von Neumann algebra. Then for any von Neumann  subalgebra $A\rtimes \G \subseteq N\subseteq (A\bar\otimes B) \rtimes \G$ one can find a $\G$-invariant von Neumann subalgebra $C\subseteq B$ so that $N= (A\bar\otimes C)\rtimes \G$.       
\end{cor}

\begin{proof} Since $\G \ca B$ is compact one can see that $\G\ca (A\subseteq A\rtimes B)$ is a compact extension and hence the conclusion follows from Theorem \ref{cptint}.\end{proof}

\noindent We end this section with an immediate application of Theorem \ref{cptint} to the study of finite index subfactors. More specifically, we show that Theorem \ref{cptint} can be used effectively to completely describe all intermediate subfactors $L(\G)\subseteq N \subseteq L^\infty(X)\rtimes \G$ with $[N:L(\G)]<\infty$ for \emph{any} free ergodic action $\G\ca X$ of \emph{any} icc group $\G$.

\begin{cor} \label{intsubgen}
Let $\Gamma$ be an icc group and let $\G \ca X$ be a free ergodic action. Let  $M=L^\infty(X)\rtimes \G$ denote the corresponding group measure space von Neumann algebra. Then the following hold:

\begin{enumerate} 
\item  Suppose $L(\Gamma) \subseteq N \subseteq L^{\infty}(X) \rtimes \Gamma $ is an intermediate von Neumann subalgebra so that $N \subseteq \mathcal {QN}_{M}(L(\Gamma))''$. Then there exists a factor $\G\ca X_0$ of $\G\ca X$ such that $N =L^\infty(X_0)\rtimes \G$  
\item For any intermediate subfactor  $L(\Gamma) \subseteq N \subseteq L^{\infty}(X) \rtimes \Gamma $ with $[N: L(\Gamma)]< \infty$ there is a finite, transitive factor $\G \ca X_0$ of $\G \ca X$ such that $N=L^\infty(X_0)\rtimes \G$; in particular $[N:L(\Gamma)] \in \mathbb{N}$. Thus for any subfactors $L(\Gamma) \subseteq N_1 \subseteq N_2 \subseteq L^{\infty}(X) \rtimes \Gamma $, with either $[N_1:L(\G)] < \infty$ or $\G \ca X$ compact,  we have $[N_2:N_1] \in \mathbb{N}\cup\{\infty\}.$
\item If $\G$ has no proper finite index subgroups (e.g.\ $\G$ is simple) then there are no nontrivial intermediate subfactors $L(\G)\subseteq N \subseteq L^{\infty}(X) \rtimes \G$ with $[N:L(\G)]<\infty$.
\end{enumerate}
\end{cor}

\begin{proof} 1. Let $\G\ca X_c$ be a maximal compact factor of $\G \ca X$ and using \cite[Theorem 6.9]{Io08a} we have that $\mathcal {QN}_{M}(L(\G))''=L^\infty(X_c)\rtimes \G $. Altogether these show that $L(\G)\subseteq N \subseteq L^{\infty}(X_c) \rtimes \Gamma$. Then the desired conclusion follows directly from Theorem \ref{cptint}. 
\vskip 0.03in
\noindent 2. Since $[N: L(\Gamma)]< \infty$ then $N$ admits a finite left (and also a finite right) Pimsner-Popa basis over $L(\G)$ and hence $N\subseteq \mathcal {QN}_M(L(\G))''$. By part 1.\ there is a factor $\G \ca X_0$ of $\G\ca X$ such that $N =L^\infty (X_0)\rtimes \G$. As $\G$ is icc and $N$ is a factor we also have that $\G \ca L^\infty(X_0)$ is ergodic. Since $N$ admits a finite Pimsner-Popa basis over $L(\G)$ then by Proposition \ref{finindex1} it follows that $\G\ca L^\infty(X_0)$ is a transitive action. In particular $X_0$ is a finite probability space and $\G\ca X_0$ is transitive. If $\G_x\leqslant \G$ is the stabilizer of an $x\in X_0$ one can also check that $[N:L(\G)]= |X_0|=|\G/\G_x|\in \mathbb N$. The rest of the statement follows easily.
\vskip 0.03in
\noindent 3. Assume that $[N: L(\Gamma)]< \infty$. From the proof of part 2.\ we have $N= L^\infty(X_0)\rtimes \G$ where $\G\ca X_0$ is an action on a finite set $X_0$ and also $[N:L(\G)]= |X_0|=|\G/\G_x|$ where $\G_x$ is the stabilizer of $x\in X_0$. Since $\G$ has no nontrivial finite index subgroups then $\G=\G_x$ and hence $N=L(\G)$.  
 \end{proof}

\noindent {\bf Final remarks.} The previous corollary also holds for intermediate subalgebras $L^\infty (X)\rtimes \G \subseteq N\subset L^\infty (Y)\rtimes \G$ with $[N: L^\infty(X)\rtimes \G]<\infty$ for von Neumann algebras arising from extensions $\G \ca L^\infty (X)\subseteq L^\infty(Y)$ of icc groups $\G$. The proof is essentially the same as the one presented in Corollary \ref{intsubgen} with the only difference that we use Theorem \ref{distaltowerCP} instead of \cite[Theorem 6.9]{Io08a}. Also, parts 1.\ and 2.\ hold for any von Neumann algebra $N$ which admits a finite Pimsner-Popa basis over $L(\G)$, if we use the Pimsner-Popa index \cite{PP86} instead of Jones index \cite{Jo81}.

\vskip 0.07in
\noindent In connection with the previous problems one may attempt to describe the subfactors of group von Neumann algebras $N\subseteq L(\G)$ that are normalized by the $\G$ itself, i.e.\ $\G \subset \mathcal N_{L(\G)}(N)$. Very recently this problem was considered in \cite{AB19} where a complete description was obtained for $\G$ lattices in higher rank simple Lie groups via a noncommutative version of Margulis' normal subgroup theorem; in turn this was obtained using character rigidity techniques introduced \cite{Pet16, CP13}. In this work we make further progress on this question for many new families of groups $\G$ complementary to the ones from \cite{AB19}. In particular, we show that under additional conditions on the relative commutant $N'\cap L(\G)$ (e.g.\ finite dimensional) these subfactors are always ``commensurable'' with von Neumann algebras arising from the normal subgroups of $\G$ (Theorem \ref{strongnormalization}). Moreover, in the case of all exact acylindrically hyperbolic groups \cite{DGO11, O16}, all nonamenable groups with positive first $L^2$-Betti number, and all lattices in product of trees the same holds without any a priori assumptions on $N'\cap L(\G)$ (see Theorem \ref{acyhypnorm}, Corollary \ref{achyp2}, and part 3 in Theorem \ref{strongnormalization}).

\begin{thm}\label{strongnormalization} Let $\G$ be a countable discrete group and let $N\subset L(\G)$ be a subfactor such that $\G\subset \mathcal N_{L(\G)}(N)$. Then there exists a normal subgroup $\La\lhd \G$ such that $N\subseteq L(\La)\subseteq N\vee (N'\cap L(\G))$. Moreover, we have the following \begin{enumerate} \item [1)] If $N'\cap L(\G)$ is finite dimensional then the inclusions $N\subseteq L(\La)\subseteq N\vee N'\cap L(\G)$ have finite index;  in particular, when $N'\cap L(\G)=\mathbb C 1$ then $N=L(\La)$.
		\item [2)]If $L(\G)$ is solid\footnote{For every diffuse $A\subseteq L(\G)$ the relative commutant  $A'\cap L(\G)$ is amenable} then either $N$ is an amenable factor or the inclusions $N\subseteq L(\La)\subseteq N\vee N'\cap L(\G)$ have finite index. Moreover if $L(\G)$ is strongly solid\footnote{For every diffuse amenable $A\subseteq L(\G)$ the normalizer $\mathcal N_{L(\G)}(A)''$ is amenable} then either $N$ is finite dimensional or the inclusions $N\subseteq L(\La)\subseteq N\vee N'\cap L(\La)$ have finite index.
		\item [3)] $\G$ be a simple group such that $L(\G)$ is a prime factor, e.g.\ Burger-Mozes group \cite{BM01}, Camm's group \cite{Ca53} or Bhattacharjee's group \cite{Bh94} (see \cite{CdW18}). Then $N$ is either finite dimensional or $[L(\G): N]<\infty$.  \end{enumerate}  
\end{thm}

\begin{proof} Denote by $\Sigma$ the set of all $\g\in \G$ for which there is $y\in \mathcal U(N)$ such that $\tau(y u_\g)\neq 0$. Note that $\Sigma$ coincides with the set of all $\g\in \G$ such that $E_N(u_\g)\neq 0$. 
	
	\noindent Fix $\g\in \Sigma$ and denote by $\phi_\g : N \ra N$ the automorphism given by $\phi_\g(x)=u_\g xu_{\g^{-1}}$ for all $x\in N$. Thus $\phi_\g(x)u_\g=u_\g x$ and applying the expectation $E_N$ we also have $\phi_\g (x)E_N(u_\g)=E_N(u_\g) x$ for all $x\in N$. These two relations give that $\phi_\g (x)E_N(u_\g) u_{\g^{-1}}=E_N(u_\g) x u_{\g^{-1}}= E_N(u_\g) u_{\g^{-1}} \phi_\g(x)$ for all $x\in N$; in particular, $a_\g:=E_N(u_\g) u_{\g^{-1}}\in N'\cap L(\G)$. Thus
	\begin{equation}\label{va1}
	E_N(u_\g)=a_\g u_\g.
	\end{equation}
	\noindent  Thus $ E_N(u_\g) E_N(u_{\g^{-1}}) = a_\g a_\g^*$. Applying the expectation $E_N$ and using $E_N\circ E_{N'\cap L(\G)}=\tau$ (since $N$ is a factor) we get $E_N(u_\g) E_N(u_{\g^{-1}})= \tau(a_\g a_\g^*) 1$.  As $a_\g\neq 0$ one can find a unitary $b_\g \in N$ so that $E_N(u_\g)=\|a_\g\|_2 b_\g$. Combining with \eqref{va1} we get $\|a_\g\|_2 b_\g= a_\g u_\g$ and hence $u_\g = \|a_\g\|_2 a_\g^* b_g$. In particular we have $\|a_\g\|_2 a_\g^*\in \mathcal U(N'\cap L(\G))$ and hence $u_\g \in \mathcal U(N)\mathcal U(N'\cap L(\G))\subseteq N\vee (N'\cap L(\G))$. Let $\La$ be the set of all $\g \in \G$ such that $u_\g=x_\g y_\g$, where $x_\g \in \mathcal U(N)$ and $y_\g\in \mathcal U(N' \cap L(\G))$. Observe that $\La \lhd\G$ is in fact a normal subgroup. The previous relations show that $\Sigma\subseteq \La$ and by the definition of $\Sigma$ we have that $N \subseteq L(\La)$. Since $L(\La)\subseteq N\vee (N'\cap L(\G))$ canonically, the first part of the conclusion follows.
	\vskip 0.02in 
	\noindent Since $N'\cap L(\G)$ is finite dimensional then  $N\vee N'\cap L(\G)$ admits left (and right) finite Pimsner-Popa basis over $N$ and \emph{1)} follows.
	\vskip 0.02in
	\noindent If $N$ is nonamenable, then $N' \cap L(\G)$ is finite dimensional, as $L(\G)$ is solid. The rest of \emph{2)} follows easily from \emph{1)}.
	\noindent If $\G$ is simple, then $\La=\G$, as $\La$ is a normal subgroup of $\G$; hence, $N \vee N' \cap L(\G)= L(\G)$. Since $L(\G)$ is prime, this further implies that either $N$ or $N' \cap L(\G)$ is finite dimensional, and thus \emph{3)} follows from \emph{1)}.\end{proof}

\noindent Next we show that whenever $\G$ is a ``negatively curved'' group then all subfactors $N\subseteq L(\G)$ normalized by $\G$ are commensurable to subalgebras $L(\La)$ arising from normal subgroups $\La\lhd \G$. Our proof relies heavily on the deformation/rigidity techniques for array/quasi-cocycles on groups that were introduced and studied in \cite{CS11,CSU11,CSU13,CKP15}. We advise the reader to consult these references beforehand.   
\vskip 0.03in 
\noindent Let $\pi: \G \ra \mathcal O(\mathcal H)$ be an orthogonal representation. Let $\mathcal{QH}_{as}^{1}(\G,\pi)$ be the set of all unbounded \emph{quasicocycles} into $\pi$, i.e.\ unbounded maps $q:\G\ra \mathcal H$ so that $d(q):=\sup_{\g,\la \in \G}\|q(\g \lambda)-q(\g)-\pi_{\g}(q(\lambda))\| < \infty$. When the defect $d(q)=0$ the set $\mathcal{QH}_{as}^{1}(\G,\pi)$ is nothing but the first cohomology group $H^1(\G,\pi)$.

\begin{thm}\label{acyhypnorm} Let $\pi: \G \rightarrow \mathcal O(\mathcal H)$ be an orthogonal mixing representation that is weakly contained in the left regular representation of $\G$. Assume one of the following holds: a) $\G$ is exact and $\mathcal{QH}_{as}^{1} (\G,\pi)\neq \emptyset$, or b) $H^1(\G ,\pi)\neq 0$. Let $N \subseteq L(\G)$ be a subfactor satisfying $\G \subset \mathcal N_{L(\G)}(N)$. Then there is a normal subgroup $\La \lhd \G$ so that $N \subseteq L(\La) \subseteq N \vee N' \cap L(\La)$  and one of the following holds:
	\begin{enumerate}
		\item $N $ is finite dimensional, or
		\item $\La$ is infinite amenable, or
		\item $[L(\La): N]<\infty$.
	\end{enumerate}
\end{thm}

\begin{proof}Let $M= L(\G)$. By Theorem \ref{strongnormalization} there is $\La \lhd \G$, such that $N \subseteq L(\La) \subseteq N \vee (N' \cap M)$ and moreover from its proof it follows that for every $\g\in \La$ there are unitaries $a_\g\in N$ and $b_\g\in N'\cap L(\La)$ so that \begin{equation}\label{nc3}u_\g =a_\g b_\g.
	\end{equation} 
	\noindent Also since $N$ is a factor, using Ge's tensor splitting result (Theorem \ref{tensorint}) we also get that \begin{equation}\label{nc1}L(\La)= N \vee (N' \cap L(\La)).
	\end{equation} 
	
	\noindent Assume that $\La$ is nonamenable. Let $q \in \mathcal{QH}_{as}^{1}(\G,\pi)$ and consider the restriction $q_{|\La}$. One can easily see that the representation $\pi_{|\La}^{\oplus \infty}$ is still mixing and is weakly contained in $\ell^2(\La)$. Moreover since $\La \lhd \G$ is normal and the representation is mixing it follows that $q|_{\La}$ is unbounded and hence $q_{|\La}\in \mathcal{QH}_{as}^{1}(\La, \pi|_{\La}^{\oplus \infty})$.	 Thus by \cite[Corollary 7.2]{CKP15} it follows that the finite conjugacy radical $FC(\La)$ of $\La$ is finite and hence $\mathcal Z(L(\La))$ is finite dimensional. 
	
	\noindent Assume that $N' \cap L(\La)$ is amenable. If it is finite dimensional then \eqref{nc1} already implies \emph{3.} If not then there is a projection $0\neq z\in \mathcal Z(N'\cap L(\La))=\mathcal Z(\La)$ such that $(N'\cap L(\La)) z$ is isomorphic to the hyperfinite factor. Since $\La$ is nonamenable $N$ is also nonamenable. Thus $L(\La)z$ has property (Gamma) and there is a sequence of $(u_n)_n$ of unitaries in $(N'\cap L(\La))z$ such that $u^\omega:=(u_n)_n \in (L(\La)'\cap L(\La)^\omega)z$ and $u^\omega \perp L(\La)$; here $\omega$ is a free ultrafilter on $\mathbb N$. On the other hand using \cite[Theorem 4.1]{CSU13} we get that $L(\La)' \cap L(\La)^{\omega} \subseteq L(\La)$. Thus $u\perp u$, which is a contradiction. 
	\vskip 0.03in 
	\noindent Now assume that $N'\cap L(\La)$ is nonamenable. If $N$ is amenable then a similar argument as before shows that $N$ is finite dimensional leading to \emph{1.} Thus for the rest of the proof we assume that $N$ and $N'\cap L(\La)$ are nonamenable and we will show this leads to a contradiction.  
	
	\vskip 0.03in
	\noindent Let $P=L(\La)$. Following \cite[Section 2.3]{CS11} consider $V_t: L^2(\tilde P) \rightarrow L^2(\tilde P)$ be the Gaussian deformation corresponding to the quasicocycle $q_{|\La}\in \mathcal QH^1_{as}(\La,\pi_{|\La}^{\oplus \infty})$ where the supralgebra $P \subset \tilde P$ is the Gaussian dilation. Let $e_P: \tilde P \rightarrow P$ denote the orthogonal projection. Since $N'\cap L(\La)$ is nonamenable there exists a nonzero projection $0\neq p\in N'\cap L(\La)$ such that $(N'\cap L(\La)) p$ has no amenable direct summand. Thus applying  a spectral gap argument a la Popa  (see for instance \cite[Theorem 3.2]{CS11}), we obtain that \begin{equation}\label{nc2}\lim_{t\ra 0} \left( \sup_{ x \in (N)_1 p}\|e_P^{\perp}V_t (x)\|_2\right) = 0, \quad \text{ and } \quad \lim_{t\ra 0} \left(\sup_{x \in (N'\cap L(\La)_1}\|e_P^{\perp}V_t (x)\|_2\right ) = 0. \end{equation} 
	
	\vskip 0.03in
	\noindent Fix $\varepsilon>0$. Thus, using the transversality property from \cite[Lemma 2.8]{CS11}, relations \eqref{nc2} and a simple calculation show that there exist $C, D>0$ satisfying  \begin{equation}\label{nc5} \begin{split}\|P_{B_C}(x)-x\|_2 < \varepsilon \text{ for all }x \in \mathcal U (N) p, \quad \quad \|P_{B_D}(y)-y\|_2 < \varepsilon  \text{ for all }y \in \mathcal U(N' \cap P).
	\end{split}\end{equation} Here for every constant $C\geq 0$ we denoted by $B_C=\{ \la \in \La \,:\,\|q(\la)\|_{\mathcal H} \leq C \}$ and by $P_{B_C}$ the orthogonal projection onto the Hilbert subspace of $L^2(\La)$ spanned by $B_C$. Since by \eqref{nc3} we have $u_{\g}=a_{\g}b_{\g}$ then \eqref{nc5} imply  
	\begin{equation}\label{nc6}\|P_{B_C}(u_{\g}b_{\g}^{\ast} p)-u_{\g}b_{\g}^{\ast} p\|_2 < \varepsilon \quad \text{ and }\quad \|P_{B_D}(b_{\g}^{\ast}p)-b_{\g}^{\ast}p\|_2 < \varepsilon \text{ for all} \g \in \La.
	\end{equation} 
	
	\noindent Thus using triangle inequality, for all $\g \in \La$, we also have 
	\begin{equation} \label{bal}
	\|P_{B_C}(u_{\g}b_{\g}^{\ast} p)- u_{\g}P_{B_D}(b_{\g}^{\ast}p)\|_2 \leq \|P_{B_C}(u_{\g}pb_{\g}^{\ast})- u_{\g}pb_{\g}\|_2+ \|P_{B_D}(b_{\g}^{\ast}p)- b^*_{\g}p\|_2< 2 \varepsilon.
	\end{equation}
	\noindent Since $q_{|\La}$ is unbounded, there exists $\g_0 \notin B_{C+D + 3 d(q_{|\La})}$. Also the quasicocycle relation and the triangle inequality show that $ B_CB_D^{-1}\subseteq B_{C+D + 3 d(q_{|\La})}$ and thus $\g_0 \notin B_CB_D^{-1}$. Hence $\langle P_{B_C}(\xi), u_{\g_0}P_{B_D}(\eta)\rangle=0$ for all $\xi, \eta \in L^2(\La)$. Thus using inequalities \ref{bal} for $\g=\g_0$ and \eqref{nc6} we see that $4 \varepsilon^2 \geq \|P_{B_C}(u_{\g_0}b_{\g_0}^{\ast}p)-u_{\g_0}P_{B_D}(b_{\g}^{\ast}p)\|_2^2 = \|P_{B_C}(u_{\g_0} b_{\g_0}^{\ast}p)\|_2^2+ \|u_{\g_0}P_{B_D}(pb_{\g_0}^{\ast})\|_2^2  \geq \|u_{\g_0}b_{\g_0}^{\ast}p\|_2^2+\|b_{\g_0}^{\ast}p\|_2^2 -2 \varepsilon^2= 2\|p\|_2^2-2\varepsilon$. Thus $\|p\|_2^2 \leq 3\varepsilon^2$, which contradicts $p \neq 0$ when $\varepsilon\ra 0$. This completes the proof of the first part of the theorem in the the case when $q$ is a quasicocycle with $d(q)\neq 0$. When $d(q)=0$  i.e.\ $q$ is a cocycle the same proof works with the only difference that to derive the convergence \ref{nc5}, instead of using \cite[Theorem 3.2]{CS11} (which requires exactness of $\G$) one can use the spectral gap arguments as in \cite{Pe09} or \cite{Va10}.   	\end{proof}
\vskip 0.03in

\noindent When combined with results in geometric group theory the previous result leads to the following

\begin{cor}\label{achyp2}  Let $\G$ be a nonamenable group that is either exact and acylindrically hyperbolic or has positive first $L^2$-Betti number. Let $N\subseteq L(\G)$ be a subfactor such that $\G\subset \mathcal N_{L(\G)}(N)$. Then there is a nonamenable normal subgroup $\La \lhd \G$ so that $N \subseteq L(\La) \subseteq N \vee N' \cap L(\La)$  and one of the following holds:
	\begin{enumerate}
		\item $N $ is finite dimensional, or
		\item $[L(\La): N]<\infty$.
	\end{enumerate}
\end{cor}

\begin{proof} From \cite{PT11} and \cite{HO11} it follows that these families always have $\mathcal QH^1_{\rm as}(\G,\ell^2(\G))\neq \emptyset$. 
	Hence the result follows directly from the previous theorem as both classes of nonamenable acylindrically hyperbolic groups and nonamenable groups with positive first $L^2$-Betti number have finite amenable radical. 
\end{proof}

%%%%%%%%%%%%%%%%%%%%%%%%%%%%%%%%%%%%%%%%%%%%%%%%%%%%%%%%%%%%%%%%%%%%%%%%%%%%%%%%%%%%%%%%%%%%%%%%%%%%%%%%%%%%%%%%%%%
%%%%%%%%%%%%%%%%%%%%%%%%%%%%%%%%%%%%%%%%%%%%%%%%%%%%%%%%%%%%%%%%%%%%%%%%%%%%%%%%%%%%%%%%%%%%%%%%%%%%%%%%%%%%%%%%%%%
%%%%%%%%%%%%%%%%%%%%%%            NESHVEYEV-STORMER CONJECTURE                  %%%%%%%%%%%%%%%%%%%%%%%%%%%%%%%%%%%
%%%%%%%%%%%%%%%%%%%%%%%%%%%%%%%%%%%%%%%%%%%%%%%%%%%%%%%%%%%%%%%%%%%%%%%%%%%%%%%%%%%%%%%%%%%%%%%%%%%%%%%%%%%%%%%%%%%
%%%%%%%%%%%%%%%%%%%%%%%%%%%%%%%%%%%%%%%%%%%%%%%%%%%%%%%%%%%%%%%%%%%%%%%%%%%%%%%%%%%%%%%%%%%%%%%%%%%%%%%%%%%%%%%%%%%

\section{Actions that satisfy Neshveyev-St\o rmer rigidity}

If $\G,\La$  are abelian (or more generally amenable) groups, and $\G \ca X$, $\La\ca Y$ are free, ergodic, pmp actions, then $L^{\infty}(X) \rtimes \G$ and $L^\infty(Y)\rtimes \La$  are isomorphic to the hyperfinite $\rm II_1$ factor $\mathcal R$. However, Neshveyev and St\o rmer proved that if we assume that $\Theta: L^{\infty}(X) \rtimes \G \rightarrow L^{\infty}(Y) \rtimes \La$ is an $\ast$-isomorphism such that $\Theta(L^{\infty}(X))$ is unitarily conjugate to $L^{\infty}(Y)$ and $\Theta(L(\G))=L(\La)$ then the actions $\G \ca X$ and $\La\ca Y$ are conjugate \cite[Theorem 4.1]{NS03}. Motivated by this group action conjugacy criterion, they further conjectured the following: if $\G,\La$ are abelian groups, $\G \ca X, \La\ca Y$ are free, weak mixing, pmp actions and $\Theta: L^{\infty}(X) \rtimes \G \rightarrow L^{\infty}(Y) \rtimes \La$ is a $\ast$-isomorphism satisfying $\Theta(L(\G))=L(\La)$ then $\G \ca X$ is conjugate to $\La\ca Y$ \cite[Conjecture]{NS03}. Shortly after, using his influential deformation/rigidity theory Popa was able to prove the following striking result: if $\G,\La$ are any countable groups, $\G \ca^\sigma X, \La\ca^\rho Y$ are free, ergodic actions, with $\sigma$ Bernoulli (or more generally clustering), and $\Theta: L^{\infty}(X) \rtimes \G \rightarrow L^{\infty}(Y) \rtimes \La$ is an $\ast$-isomorphism such that $\Theta(L(\G))=L(\G)$ then $\G \ca X$ is conjugate to $\La\ca Y$, \cite[Theorem 5.2]{Po04}. In particular, this settled Neshveyev-St\o rmer conjecture for Bernoulli actions. Popa also showed that the study of the Neshveyev-St\o rmer rigidity question in the context of icc property (T) groups eventually leads to his remarkable proof of the \textit{group measure space} version of Connes' rigidity conjecture, \cite[Theorem 0.1]{Po04}. All these results motivate the study of the following \textit{generalized Neshveyev-St\o rmer rigidity question}.
\begin{question}\label{genNS}
	Let $\G$ and $\La$ be icc countable discrete groups and let $\G \ca X$ and $\La \ca Y$ be free, ergodic, pmp actions. Assume that there is a $\ast$-isomorphism $\Theta: L^{\infty}(X) \rtimes \G \rightarrow L^{\infty}(Y) \rtimes \La$ so that $\Theta(L(\G))=L(\La)$. Under what conditions on $\G \ca X$ can we conclude that $\G \ca X$ and $\La \ca Y$ are conjugate?
	\end{question}
\noindent Informally, the generalized Neshveyev-St\o rmer rigidity question asks, under what conditions can $\G \ca X$  be completely recovered from the irreducible subfactor inclusion $L(\G) \subset L^{\infty}(X) \rtimes \G$.

\noindent Using existing literature, one can that the generalized Neshveyev-St\o rmer rigidity phenomenon holds for the following classes of actions: all Bernoulli actions of icc groups, \cite{Po04}; all $W^*$-superrigid actions, \cite{Pe09,PV09,CP10,Va10,Io10,HPV10,CS11,CSU11,PV11,PV12,Bo12,CIK13,CK15,Dr16}; all weak mixing $\mathcal C_{\rm gms}$-superrigid  actions, \cite[Theorem 5.1]{Po04}; and all mixing Gaussian actions \cite[Corollary 3.9]{Bo12}.  

\noindent In this section we provide new classes of actions satisfying the generalized Neshveyev-St\o rmer question, most notably, all actions that appear as (nontrivial) mixing extension of free distal actions (see Theorem \ref{main2}).

\subsection{A criterion for conjugacy of group actions} 

%\noindent Neshveyev-St\o rmer established in \cite[Theorem 4.1 ]{NS03} a criterion for conjugacy of group actions on probability spaces. Specifically, they showed that if $\G$ and $\La$ are abelian groups, $\G\ca X$ is a weak mixing action, $\La \ca Y$ is any action, and $\theta:L^\infty(X)\rtimes \G\ra L^\infty(X)\rtimes \La$ is an $\ast$-isomorphism such that $\Theta(L(\G))=L(\La)$ and the Cartan subalgebras $\Theta(L^\infty(X))$ and $L^\infty(Y)$ are unitary conjugate then $\G\ca X$ and $\La\ca Y$ are conjugate in a way compatible with $\Theta$. Using his influential deformation/rigidity techniques Popa was able to obtain a far-reaching generalization of this result within the class of clustering actions (e.g. Bernoulli) by completely dropping the assumption on conjugacy of the Cartan subalgebras. In particular, Popa's result confirmed Neshveyev- St\o rmer rigidity question for all Bernoulli actions.   
%\vskip 0.04in
\noindent Within the class of icc group, we further generalize Neshveyev-St\o rmer's aforementioned criterion for conjugacy of group actions on probability spaces by completely removing the weak mixing assumption of $\G \ca X$ (see Theorem \ref{3'}). In this context our result also generalizes \cite[Theorem 0.7]{Po04} as it covers many new actions (e.g.\ compact) that were not previously analyzed in this context. Our proof relies on the usage of the notion of height of elements in group von Neumann algebras introduced in \cite{IPV10}. In order to prove our result we need to establish first a few preliminary technical results on height of elements in group von Neumann algebras, \cite[Definition 3.1]{IPV10}.     
\vskip 0.03in

\begin{defn} A trace preserving action $\G \ca^\sigma A$ on a finite von Neumann algebra $A$ is  called properly outer over the the center of $A$ if for every $\g\neq 1$ and every $0\neq z\in \mathcal Z(A)$ such that $\sigma_\g(z)=z$ the automorphism $\sigma_\g : Az \ra Az$ is not inner. When $A$ is abelian this amounts to the usual freeness of the action $\G \ca A$. \end{defn}

\noindent The following lemma is a basic generalization of Dye's famous result in the case of group measure space von Neumann algebras. For readers' convenience we include a short proof.

 \begin{lem} \label{Dye}
	Let $\G \ca^{\sigma} A$ and $\La \ca^{\alpha} B$ be properly outer actions. Also let $\Theta: A \rtimes \G \ra B \rtimes \La $ be a $\ast$-isomorphism such that $\Theta(A)=B$. Fix $\g \in \G$ and let $\Theta(u_\g)= \sum_{\lambda \in \La}a_{\lambda}v_{\lambda}$ be the Fourier decomposition of $\Theta(u_g)$ in $B \rtimes \La $. Then there are mutually orthogonal projections $\{e_{\lambda}\}_{\lambda \in \La} \subset  \mathcal Z (B)$ and unitaries $\{x_{\lambda}\}_{\lambda \in \La} \subset B$ so that $a_{\lambda} = e_{\lambda}x_{\lambda}$ for all $\lambda \in \La$. Also, $\sum_{\lambda \in \La}e_{\lambda}=1$.
\end{lem}

\begin{proof} To ease our presentation we assume that $A=B$. Thus, $M:=A \rtimes \G= A \rtimes \Lambda$. Fix $\g \in \G$ and let $u_\g= \sum_{\lambda \in \La}a_{\lambda}v_{\lambda}$. Since $\sigma_\g(a)u_\g=u_\g a$ for all $a \in A$ then $\sigma_\g(a)\sum_{\lambda \in \La}a_{\lambda}v_{\lambda} =\sum_{\lambda \in \La}a_{\lambda}v_{\lambda}a= \sum_{\lambda \in \La}a_{\lambda}\alpha_{\lambda}(a)v_{\lambda }$. Thus  $\sigma_{\g}(a) a_{\lambda} = a_{\lambda}\alpha_{\lambda}(a)$ for all $a\in A$ and $\la\in \Lambda$. If $a_\la = w_\la |a_\la|$ is the polar decomposition of $a_\la$ this further implies that for all $a\in A$ and $\la\in \Lambda$ we have  
\begin{equation}\label{3.2.1}\sigma_{\g}(a) w_{\lambda} =w_{\lambda} \alpha_{\lambda}(a).
\end{equation}
\noindent Hence  $e_\la=w_\la w^*_\la\in \mathcal Z(B)$. Let $x_\la\in \mathcal U(A)$ such that $w_\la = x_\la e_\la$.  Fix $\lambda \neq \mu \in \La$. Using \eqref{3.2.1}, for all $a \in A$  we have $\sigma_{\g}(a)e_{\lambda}= x_\la \alpha_{\lambda}(a)x^*_{\lambda}e_\la$ and  $\sigma_{\g}(a)e_{\mu}= x_\mu\alpha_{\mu}(a)x^*_{\mu}e_\mu$. Thus $\sigma_{\g}(a)e_{\lambda}e_{\mu}= x_\la \alpha_{\lambda}(a)x_\la^*e_{\lambda}e_{\mu}=x_\mu\alpha_{\mu}(a) x_\mu^*e_{\lambda}e_{\mu}$.
Letting $a=\alpha_{\mu^{-1}}(b)$ we get $\alpha_{\lambda \mu^{-1}}(b)e_{\lambda}e_{\mu}=x_\la^*x_\mu be_{\lambda}e_{\mu}x_\mu^*x_\la$ for all $b \in A$. Also one can easily check that $\alpha_{\la\mu^{-1}}(e_\la e_\mu)=e_\la e_\mu$.  
Since  $\La \ca^{\alpha} A$ is properly outer and $\lambda \mu^{-1} \neq 1$, we get  $e_{\lambda} e_{\mu}=0$; thus for all $\lambda \neq \mu$ we have $e_\la e_\mu=0$.

\noindent As $u_\g\in \mathcal U(M)$ we have $1 = \sum_{\lambda} a_{\lambda}^{\ast}a_{\lambda}=|a_{\lambda}|^2 \leq \sum_{\lambda}e_{\lambda} \leq 1.$ Thus $|a_{\lambda}|^2 = e_{\lambda}$ and hence $|a_{\lambda}| =e_{\la}$ for all $\lambda \in \La$; moreover $\sum_\lambda e_\lambda =1$.\end{proof}

\noindent With this result at hand we are now ready to prove the first technical result needed in the proof of Theorem \ref{3'}.

\begin{thm}\label{2'} Let $\G \ca^{\sigma} A$ and $\La \ca^{\alpha} B$ be properly outer actions. Assume that  $\Theta: A\rtimes \G \ra B\rtimes \La$ is a $\ast$-isomorphism satisfying the following conditions: 
	\begin{itemize}
		\item [i)] $\Theta(A)=B$, and
		\item [ii)] there exist $1> \varepsilon >0$ and a finite subset $K \subseteq B \rtimes \Lambda$ such that for every $\g \in \G$ we have  \begin{equation}\|\Theta(u_\g)- \sum_{a,b,c,d \in K} aE_{L(\La)}(b(\Theta(u_\g))c)d\|_2 \leq \varepsilon.\end{equation} 
	\end{itemize}  
Then one can find  $D>0$ and finite subset $F \subseteq B$ such that for every $\g \in \G$ there exists $\lambda \in \La$ satisfying $\max_{b,c \in F}|\tau(b^{\ast}\Theta(u_\g)cv_{\lambda})| \geq D>0$.
\end{thm}

\begin{proof} As before assume that $A=B$ and notice that $M=A \rtimes \G= A \rtimes \Lambda$. Let $1> \varepsilon >0$ and $K \subseteq M$ a finite subset such that for all $g \in \G$ we have
	\begin{equation}\|u_\g - \sum\limits_{z,t,w,r \in K} z E_{L(\La)}(tu_\g w)r\|_2 \leq \varepsilon.\end{equation}
	Approximating the elements of $K$ via Kaplansky's density theorem we can assume there are finite subsets $F \subseteq A$, $G \subseteq \La$ (some elements could be repeated finitely many times!) so that for all $\g \in \G$ we have \begin{equation*} \varepsilon\geq \|u_\g - \sum\limits_{\substack{a,b,c,d \in F \\
		 \la_1,\la_2,\la_3,\la_4 \in G}}a v_{\la_1}E_{L(\La)}(v_{\la_2^{-1}}b^{\ast}u_\g cv_{\la_3})v_{\la_4^{-1}}d^{\ast}\|^2_2. \end{equation*}

	 \noindent For the simplicity of writing we convene for the rest of the proof that $\sum_{\substack{a,b,c,d \in F\\
	 		\la_1,\la_2,\la_3,\la_4 \in G}}=\sum_{F, G}$. Thus for all $\g \in \G$ we have

	\begin{equation}\label{3.5.4} \begin{split}\varepsilon &\geq \|u_\g - \sum\limits_{\la \in \La}\sum\limits_{F,G}a \tau(v_{\la_1\la_2^{-1}}b^{\ast}u_\g cv_{\la_3 \la_4^{-1}\la^{-1}})v_{\la}d^{\ast}\|^2_2 \\ 
	&\geq \sum\limits_{\la \in \La}(\|E_A(u_\g v_{\la^{-1}}) -  \sum\limits_ {F, G} \tau(v_{\la_1\la_2^{-1}}b^{\ast}u_\g cv_{\la_3 \la_4^{-1}\la^{-1}})a\alpha_{\la}(d^{\ast})\|^2_2). \end{split}
	\end{equation}
\noindent By the previous lemma \ref{Dye} we have that $E_A(u_\g v_\la)= e_\lambda x_\la$ with $x_\lambda \in \mathcal U(A)$ and $e_\la \in \mathcal Z(A)$. 
Then using $ \| |f|- |g| \|_2 \leq \| f-g\|_2 $ for  $f,g\in A$ we see that the last quantity in \eqref{3.5.4} is larger than 
\begin{equation}\label{3.5.5} \begin{split}& \geq \sum\limits_{\la \in \La} (\| e_{\la} - | \sum\limits_{ F, G} \tau(v_{\la_1\la_2^{-1}}b^{\ast}u_\g cv_{\la_3 \la_4^{-1}\la^{-1}})a\alpha_{\la}(d^{\ast})| \|_2^2) \\				
& =\sum\limits_{\la \in \La}(\|e_{\la}\|_2^2 + \| |\sum\limits_{ F,  G} \tau(v_{\la_1\la_2^{-1}}b^{\ast}u_\g cv_{\la_3 \la_4^{-1}\la^{-1}}) a\alpha_{\la}(d^{\ast})| \|_2^2 - \\
& \quad -2 Re \tau(e_{\la}| \sum\limits_{ F,  G } \tau(v_{\la_1\la_2^{-1}}b^{\ast}u_\g cv_{\la_3 \la_4^{-1}\la^{-1}})a\alpha_{\la}(d^{\ast})|)).\end{split}
\end{equation}
					
\noindent From Lemma \ref{Dye} we also have that $\sum_{\la \in \La}e_{\la}=1$ and hence  $ \sum_{\la \in \La}\|e_{\la}\|_2^2=1$. Combining this with \eqref{3.5.4} and \eqref{3.5.5} we get  
\begin{equation*}\begin{split}
\sum\limits_{\la \in \La}\varepsilon \|e_{\la}\|_2^2 \geq  \sum_\lambda \|e_\la\|_2^2+\| |\sum\limits_{F,  G} \tau(v_{\la_1\la_2^{-1}}b^{\ast}u_\g cv_{\la_3 \la_4^{-1}\la^{-1}})a\alpha_{\la}(d^{\ast})|\|_2^2 &\\ \quad - 2 Re \tau(e_{\la}| \sum\limits_{ F,  G }\tau(v_{\la_1\la_2^{-1}}b^{\ast}u_\g cv_{\la_3 \la_4^{-1}\la^{-1}}) a\alpha_{\la}(d^{\ast})|)).
\end{split}\end{equation*}

\noindent Hence, for every $\g \in \G$ there exists $\la \in \La$ such that $e_{\la} \neq 0$ satisfies
\begin{equation}\label{3.5.6}
\begin{split}
\varepsilon \|e_{\la}\|_2^2 &\geq \|e_\la\|_2^2+\| \sum\limits_{ F,  G} \tau(v_{\la_1\la_2^{-1}}b^{\ast}u_\g cv_{\la_3 \la_4^{-1}\la^{-1}})a\alpha_{\la}(d^{\ast})\|_2^2 -\\ &\quad  -	2 Re \tau(e_{\la}| \sum\limits_{ F, G} \tau(v_{\la_1\la_2^{-1}}b^{\ast}u_\g cv_{\la_3 \la_4^{-1}\la^{-1}})a\alpha_{\la}(d^{\ast})|).
\end{split}
\end{equation}

\noindent Using \eqref{3.5.6} and the operatorial inequality $(\sum^n_{i=1}  x_i)^*(\sum^n_{i=1} x_i)\leq 2^{n-1}\sum_{i=1}^n x^*_ix_i$ we get
\begin{equation}
\begin{split}
(2- 2\sqrt\varepsilon)\|e_{\lambda}\|_2^2 &\leq (1-\varepsilon) \|e_\la \|^2_2 + \| | \sum\limits_{F, G} \tau(v_{\la_1\la_2^{-1}}b^{\ast}u_\g cv_{\la_3 \la_4^{-1}\la^{-1}})a\alpha_{\la}(d^{\ast})| \|_2^2 \\
&\leq 2 Re \tau(e_{\la}| \sum\limits_{F, G }\tau(v_{\la_1\la_2^{-1}}b^{\ast}u_\g cv_{\la_3 \la_4^{-1}\la^{-1}})a \alpha_{\la}(d^{\ast})|) \\
& \leq 2^{|F|^4|G|^4+1} Re \tau(e_{\la} (\sum\limits_{F, G }|\tau(v_{\la_1\la_2^{-1}}b^{\ast}u_\g cv_{\la_3 \la_4^{-1}\la^{-1}})|^2 |a \alpha_{\la}(d^{\ast})|^2)^{1/2}) \\
&\leq 2^{|F|^4|G|^4+1} Re \tau( (\sum\limits_{F, G }|\tau(v_{\la_1\la_2^{-1}}b^{\ast}u_\g cv_{\la_3 \la_4^{-1}\la^{-1}})|^2 \|a\|^2_\infty \|d\|_\infty ^2)^{1/2} e_{\la})\\
& \leq 2^{|F|^4|G|^4+1} (\max_{a \in F}\|a\|_{\infty})^2|G|^4|F|^4 \max_{\mu \in GG^{-1}\la^{-1}GG^{-1}, b, c \in F}|\tau(b^{\ast}u_\g c v_{\mu})| \|e_\la\|^2_2. 
\end{split}
\end{equation}

\noindent Letting $0<D_0:= 2^{|F|^4|G|^4+1} (\max_{a \in F}\|a\|_{\infty})^2|G|^4|F|^4$ and using that $\|e_{\la}\|_2 \neq 0$, the previous equation gives that $\max_{\mu \in GG^{-1}\la^{-1}GG^{-1}, b, c \in F}|\tau(b^{\ast}u_\g c v_{\mu})| \geq \dfrac{1-\sqrt\varepsilon}{D_0}>0$, which finishes the proof.\end{proof}

\noindent The previous technical result on height can be successfully exploited in combination with some soft analysis arising from icc property for groups in order to derive the conjugacy criterion for actions.

\begin{thm}\label{3'} Let $\G \ca X$ and $\La \ca Y$ be free ergodic actions where $\G$ is icc. Assume that  $\Theta: L^\infty(X)\rtimes \G \ra L^\infty(Y)\rtimes \La$ is a $\ast$-isomorphism such that $\Theta(L^\infty(X))=L^\infty(Y)$ and there exists a unitary $u\in L^\infty(Y)\rtimes \La$ such that $\Theta (L(\G))= u L( \La)u^*$. Then one can find $x\in \mathcal N_{L^\infty(Y)\rtimes \La}{(L^\infty (Y))}$, a character $\eta: \G \ra \mathbb T$, and a group isomorphism $\delta:\G\ra \La$ such that $xu\in \mathcal U(L(\La))$ and for all $a\in L^\infty(X)$, $\g\in \G$ we have \begin{equation}\label{conj}
	\Theta(au_\g)= \eta(\g) \Theta(a) x^* v_{\delta(\g)} x.
	\end{equation} 
	\noindent Here $\{u_\g\}_{\g\in \G}$  and $\{v_\la \}_{\la\in \La}$ are the canonical group unitaries implementing the actions in $L^\infty(X)\rtimes \G$ and $L^\infty(Y)\rtimes \La$, respectively.
	
	\noindent In particular, it follows that $\G\ca X$ is conjugate to $\La \ca Y$.\end{thm}

\begin{proof}For the ease of presentation we first introduce some notations. After suppressing $\Theta$ from the notation  we assume that $A=L^\infty (X)=L^\infty(Y)$ and hence $M= A \rtimes_\sigma \G = A\rtimes_\alpha\La$. Also letting $C=uAu^*$ and $\La_1=u\La u^*$ we also have $M= C \rtimes_{\alpha'} \La_1$ and $L(\G)=L(\La_1)$. Throughout the proof we denote by  $t_\la = u v_\la u^*$ and $\alpha'_\la (c)=t_\la c t_{\la^{-1}}$ for all $c\in C$.   Note that the condition ii) in Theorem \ref{2'} is automatically satisfied  and hence by the conclusion of Theorem \ref{2'} there exists a $D>0$ and a finite subset $F \subset A$ so that for every $\g \in \G$, there is $\la \in \La$ such that 
	\begin{equation}\label{10'}
	 D\leq \max_{b,c \in F}|\tau(b^{\ast}u_\g cv_{\la})|=\max_{b,c \in F}|\tau((u b^{\ast})u_\g(c u^*) uv_{\la}u^*)|=\max_{b,c \in F}|\tau((u b^{\ast})u_\g(c u^*) t_{\la})|.
	\end{equation}
\noindent Approximating $b^{\ast}u$, $u^{\ast}c$ in \eqref{10'} via Kaplansky's theorem with elements in $C\rtimes_{\rm alg} \La_1$ and then diminishing $D$ if necessary, we can in fact assume the following:  there exists $D>0$ and  $K \subset C$ finite, such that for every $\g \in \G$, there exists $\la \in \La_1$ satisfying 
\begin{equation}\label{11'}\max_{d,e \in K}|\tau(d u_\g et_{\la})| \geq D.
\end{equation} On the other hand since $E_C(x)=\tau(x)1$ for all $x\in L(\La_1)$ we can see that 
\begin{align*}
\max_{d,e\in K}|\tau(d u_{\g} e t_{\la})|  & =\max_{d,e\in K} |\tau( du_\g t_\la \alpha'_{\la^{-1}}(e))| =\max_{d,e\in K}|\tau( E_C(du_\g t_\la \alpha'_{\la^{-1}}(e)))| \\ &= \max_{d,e\in K}|\tau( d E_A(u_\g t_\la) \alpha'_{\la^{-1}}(e))|=\max_{d,e\in K}|\tau(u_\g v_\la)||\tau( d \alpha'_{\la^{-1}}(e))|\\ 
&\leq  \max_{d\in K}\|d\|^2_\infty |\tau(u_\g t_\la)|.
\end{align*}
Combining this with \eqref{11'}, for every $\g \in \G$ there exists $\la_1 \in \La_1$ such that $|\tau(u_\g t_{\la_1})| \geq \dfrac{D}{\max_{d \in K}\|d\|^2_{\infty}}>0.$ Since $L(\G)=L(\La_1)$ and $h_{\La_1}(\G) >0$ then by \cite[Theorem 3.1]{IPV10} there is $w\in \mathcal U(L(\La_1))$, a character $\eta:\G\ra \mathbb T$, and a group isomorphism  $\delta_1 : \G\ra \La_1$ satisfying $wu_\g w^*= \eta(\g) t_{\delta_1(\g)}$. Since $t_\la = u v_\la u^*$, letting $x=u^*w$, we further get that there is a group isomorphism  $\delta : \G\ra \La$ satisfying 
		\begin{equation}\label{gpisom} x u_\g x^*= \eta(\g) v_{\delta(\g)}, \text{ for all } \g \in \G.
	\end{equation} 
	\vskip 0.06in
	\noindent As $v_\la  A v_{\la^{-1}}=A$, using \eqref{gpisom} we get $x u_h x^*A xu_{h^{-1}} x^*=A$ for all $h\in \G$. Fix arbitrary $a\in A$ and note $u_hx^* a xu_{h^{-1}}= x^*E_A(xu_hx^* a xu_{h^{-1}} x^*) x$. Applying the expectation we also have $u_hE_A(x^* a x)u_{h^{-1}}= E_A(x^*E_A(xu_hx^* a xu_{h^{-1}} x^*) x)$. Subtracting these relations, for every $h\in \G$ we have 
	\begin{equation}\label{17}
	u_h(x^*ax-E_A(x^* a x))u_{h^{-1}}= x^*E_A(xu_hx^* a xu_{h^{-1}} x^*) x-E_A(x^*E_A(xu_hx^* a xu_{h^{-1}} x^*) x).
	\end{equation} 
	
	\noindent Fix $\varepsilon>0$. By Kaplansky Density Theorem there exist finite subsets $K \subset \G\setminus\{1\}$, $L\subset \G$ and elements $y_K\in {\rm span} A K$ and $x_L\in {\rm span} A L$ such that 
	
	\begin{equation}\label{18}\begin{split}
	&\|y_K\|_\infty \leq 2,\quad \|x_L\|_\infty \leq 1\\
	&\| x^*ax-E_A(x^* a x) - y_K\|_2\leq \varepsilon,\quad \|x-x_L\|_2\leq \varepsilon
	\end{split}
	\end{equation}
	Using \eqref{17} and \eqref{18} together with basic calculations we see that for every $h\in \G$  we have  
	
	\begin{equation}\label{19}
	\begin{split}
	&\|x^*ax-E_A(x^* a x)\|^2_2=\|u_h(x^*ax-E_A(x^* a x))u_{h^{-1}}\|^2_2\\
	&= |\langle u_h(x^*ax-E_A(x^* a x))u_{h^{-1}}, x^*E_A(xu_hx^* a xu_{h^{-1}} x^*) x-E_A(x^*E_A(xu_hx^* a xu_{h^{-1}} x^*) x) \rangle|\\
	&\leq 2\varepsilon+|\langle u_h y_K u_{h^{-1}}, x^*E_A(xu_hx^* a xu_{h^{-1}} x^*) x-E_A(x^*E_A(xu_hx^* a xu_{h^{-1}} x^*) x) \rangle|\\
	&\leq 2\varepsilon+|\langle u_h y_K u_{h^{-1}}, x^*E_A(xu_hx^* a xu_{h^{-1}} x^*) x \rangle|\\
	&\leq 4\varepsilon+|\langle u_h y_K u_{h^{-1}}, {x_L}^*E_A(xu_hx^* a xu_{h^{-1}} x^*) x_L \rangle|\\
	\end{split}
	\end{equation}
	Since $\G$ is icc and $K\subset \G\setminus \{1\}$, $L\subset \G$ are finite then by \cite[Proposition 2.4]{CSU13} there is $h\in \G$ so that $hKh^{-1}\cap L^{-1}L =\emptyset$. Hence  $\langle u_h y_K u_{h^{-1}}, {x_L}^*E_A(xu_hx^* a xu_{h^{-1}} x^*) x_L \rangle=0$ and using \eqref{19} we conclude that $\|x^*ax-E_A(x^* a x)\|^2\leq 4\varepsilon$. Since this holds for all $\varepsilon>0$ then $x^*ax=E_A(x^* a x)$ for all $a\in A$. Therefore $x^* Ax\subseteq A$ and since $A$ is a MASA we obtain $x^* Ax = A$; thus $x\in \mathcal N_M(A)$. This together with \eqref{gpisom} give \eqref{conj}. In addition, for every $a\in A$ and $\g\in \G$ we have $x\sigma_\g(a)x^*= xu_\g au_{\g^{-1}}x^*= v_{\delta(\g)} xax^* v_{\delta(\g)^{-1}}= \alpha_{\delta(\g)}(xax^*)$; in particular $\G\ca X$ and $\Lambda\ca Y$ are conjugate. 
\end{proof} 
\vskip 0.03in

\noindent{\bf Remarks.} The Theorem \ref{3'} actually holds in a greater generality, namely, for all actions $\G\ca A$, $\La\ca B$ that are properly outer over the center. The proof is essentially the same with the one presented above. We highlighted only the more particular case of free ergodic actions solely because this is what we will mainly use to derive the main results of this section.   
\vskip 0.09in

\subsection{Applications to the generalized Neshveyev-St\o rmer rigidity question}
\noindent  In this subsection we show that large families of group actions verify the conjugacy criterion presented in Theorem \ref{3'} and therefore will satisfy the generalized Neshveyev-St\o rmer rigidity question. Our examples appear as mixing extensions of free distal actions. Our method of proof rely on combining the persistence of mixing through von Neumann equivalence from Section \ref{prelim-mixing} and the von Neumann algebraic description of compactness using quasinormalizers from \cite{Ni,Pa,NS03,Io08a,CP18}.

\begin{thm}\label{1} Let $\Gamma\ca X$ be a ergodic pmp action whose distal quotient  $\G \ca X_d $ is free and the extension $\pi: X \rightarrow X_d$ is nontrivial and mixing. Let $\Lambda\ca Y$ be an ergodic pmp action whose distal quotient $\Lambda\ca Y_d$ is also free. Assume that $\Theta: L^\infty(X)\rtimes \G \ra L^\infty(Y)\rtimes \Lambda$ is a $\ast$-isomorphism such that $\Theta(L(\G))=L(\La)$. Then there exists a unitary $u\in L^\infty(Y_d)\rtimes \Lambda$ such that $\Theta(L^\infty(X_d)) =  uL^\infty (Y_d)u^*$ and  $\Theta(L^\infty(X)) =  uL^\infty (Y)u^*$.  
\end{thm}

\begin{proof} To ease our presentation we assume that $M:= L^\infty(X)\rtimes \G= L^\infty(Y)\rtimes \Lambda$ with $P=L(\Gamma)=L(\La)$. Using Theorem \ref{distaltowerCP} it follows that $N:=L^\infty(X_d)\rtimes \G=L^\infty(Y_d)\rtimes \La$. Next we argue that $L^\infty(Y_d)\prec_N L^\infty(X_d)$. Indeed, if we assume $L^\infty(Y_d)\nprec_N L^\infty(X_d)$, since the extension $\pi:X\ra X_d$ is assumed to be mixing, by Theorem \ref{mixingextn} we have that $\mathcal{QN}_M(L^\infty(Y_d))''\subseteq N$. However since $\mathcal{QN}_M(L^\infty(Y_d))''=M$ it would imply that $M\subseteq N$ which is a contradiction.
	
	\noindent Since $\G \ca X_d$ and $\La\ca Y_d$ are free and $L^\infty(Y_d)\prec_N L^\infty(X_d)$ then by \cite[Appendix A]{Po01} one can find a unitary $u\in N$ so that  $L^\infty(X_d) =  uL^\infty (Y_d)u^*$. Passing to relative commutants and using freeness of $\G\ca X_d$, $\La\ca Y_d$ again we also get  $L^\infty(X) =  uL^\infty (Y)u^*$, as desired. \end{proof}

\begin{thm}\label{1'} Let $\G$ be an icc group and let $\Gamma\ca X$ be an ergodic pmp action whose distal quotient $\G \ca X_d $ is free and the extension $\pi: X \rightarrow X_d$ is nontrivial and mixing. Let $\Lambda\ca Y$ be any free ergodic pmp action. Assume that $\Theta: L^\infty(X)\rtimes \G \ra L^\infty(Y)\rtimes \Lambda$ is a $\ast$-isomorphism such that $\Theta(L(\G))=L(\La)$. Then there exists a unitary $u\in L^\infty(Y_d)\rtimes \Lambda$ such that $\Theta(L^\infty(X_d)) =  uL^\infty (Y_d)u^*$ and  $\Theta(L^\infty(X)) =  uL^\infty (Y)u^*$.  
\end{thm}

\begin{proof} As before we assume that $M:= L^\infty(X)\rtimes \G= L^\infty(Y)\rtimes \Lambda$ with $L(\Gamma)=L(\La)$. Using Theorem \ref{distaltowerCP} it follows that $N:=L^\infty(X_d)\rtimes \G=L^\infty(Y_d)\rtimes \La$.  Next we argue that $L^\infty(X_d)\prec_N L^\infty(Y_d)$. First notice since $\pi:X\ra X_d$ is mixing it follows from Theorem \ref{mix2} that $\pi:Y\ra Y_d$ is also mixing.  If we would have $L^\infty(Y_d)\nprec_N L^\infty(X_d)$, since the extension $\pi:Y\ra Y_d$ is mixing, then Theorem \ref{mixingextn} would imply that $\mathcal{QN}_M(L^\infty(X_d))''\subseteq N$. However since $\mathcal{QN}_M(L^\infty(X_d))''=M$ it would imply that $M\subseteq N$ which is a contradiction. 

\noindent Notice that since $\G$ is icc and $L(\G)=L(\La)$ it follows that $\La$ is icc as well. Also since $L^\infty(X_d)\prec_N L^\infty(Y_d)$ and  $L^\infty (X)$ is a Cartan subalgebra in $M$ it follows from \cite[Lemma 4.1]{OP07} that $\La \ca Y_d$ is free and then the desired conclusion follows from Theorem \ref{1}.\end{proof}

\noindent {\bf Remarks.} 1) If in the statements of Theorems \ref{1} and \ref{1'} one only requires that the distal factor $\G \ca X_d$ is actually compact, then in the proof of Theorem \ref{1} we don't need to use Theorem \ref{distaltowerCP}. Instead one can just directly apply \cite[Proposition 6.10]{Io08a}. 

\vskip 0.04in
\noindent 2) If in the statement of Theorem \ref{1'} one requires that the first element $\G \ca X_0$ of the distal tower $\G\ca X_d$ is free profinite then one can show the action $\La\ca Y_d$ is free without appealing to \cite[Lemma 4.1]{OP07}. Briefly, using the mixing we have $L^\infty (X_0)\prec_M L^\infty(Y)$ and employing some basic intertwining properties one can further show that $L^\infty (X_0)\prec_{L^\infty(X_0)\rtimes \G} L^\infty(Y_0)$ and hence $L^\infty(Y_0)'\cap (L^\infty(Y_0)\rtimes \La) \prec_{L^\infty(Y_0)\rtimes \La} L^\infty(X_0)$ ($\ast$). However using the same calculations from the proof of part 2.\ in Theorem \ref{profnsprob} we have $L^\infty(Y_0)'\cap (L^\infty(Y_0)\rtimes \La)= L^\infty(Y_0)\rtimes \Sigma$ for some normal subgroup $\Sigma\lhd \La$. However since $L(\Sigma)\subseteq L(\G)$ the the intertwining ($\ast$) implies that $\Sigma$ is finite and since $\La$ is icc we further have $\Sigma=1$; hence $\La \ca Y_0$ must be free.
\vskip 0.04in
\noindent Combining the previous theorems with Theorem \ref{3'} we obtain the following 

\begin{thm}\label{main2}
	Let $\G$ be an icc group and let $\Gamma\ca X$ be a free, ergodic pmp action whose distal quotient $\G \ca X_d $ is free and the extension $\pi: X \rightarrow X_d$ is nontrivial and mixing. Let $\Lambda\ca Y$ be any free ergodic pmp action. Assume that $\Theta: L^\infty(X)\rtimes \G \ra L^\infty(Y)\rtimes \Lambda$ is a $\ast$-isomorphism such that $\Theta(L(\G))=L(\La)$. Then there exists $y \in \mathcal U(L(\La)$, $\omega: \G \rightarrow \mathbb{T}$ a character, and $\delta: \G \rightarrow \La$ a group isomorphism such that $y\Theta(L^{\infty}(X))y^{\ast}=L^{\infty}(Y)$, and for all $a \in L^{\infty}(X), \g \in \G$, we have $$\Theta(au_{\g})=\omega(\g)\Theta(a) y^{\ast}v_{\delta(\g)}y.$$ In particular, we have $y \Theta (\sigma_\g(a))y^*= \alpha_{\delta(\g)}(y\Theta(a)y^*)$ and hence $\G\ca X$ and $\La\ca Y$ are conjugate. 
\vskip 0.02in	
	\noindent Here $\{u_\g\}_{\g\in \G}$  and $\{v_\la \}_{\la\in \La}$ are the canonical group unitaries implementing the actions in $L^\infty(X)\rtimes \G$ and $L^\infty(Y)\rtimes \La$, respectively.

\end{thm}

\begin{proof}
	To ease our presentation, we assume, as before, that $L^{\infty}(X) \rtimes \G = L^{\infty}(Y) \rtimes \La$, and $L(\G)= L(\La)$. Theorem \ref{1'} yields that there is a unitary $u \in L^{\infty}(Y_d) \rtimes \La$ such that $L^{\infty}(X)=uL^{\infty}(Y)u^{\ast}$. 	This is equivalent to assuming that $L^{\infty}(X) \rtimes \G = L^{\infty}(Y) \rtimes \La$,  $L(\G)= uL(\La)u^{\ast}$ and $L^{\infty}(X)=L^{\infty}(Y)$.
	We are now exactly in the set up of Theorem \ref{3'}, which yields the desired conclusions. \end{proof}

\noindent{\bf Examples.} Theorem \ref{main2} implies that if $\G$ be an icc group and $\Gamma\ca X$ is any ergodic pmp action that admits a free profinite quotient $\G \ca X_d $ and the extension $\pi: X \rightarrow X_d$ is nontrivial and mixing then $\G\ca X$ satisfies Neshveyev-St\o rmer rigidity question. For instance if $\G$ is icc residually finite then this is the case for any diagonal action $\G \ca Z\times T$ where $\G\ca Z$ is a Gaussian action associated to a mixing orthogonal representation of $\G$ and $\G\ca T$ is any free ergodic profinite action.

\vskip 0.05in

\begin{cor}\label{mixingint} Let $\G$ be an icc group, let $\Gamma\ca X$ be a free, mixing pmp action and let $\Lambda\ca Y$ be any free ergodic pmp action. Also let $\G \ca X_0$ be a free factor of  $\G\ca X$ and $\La \ca Y_0$ be a factor of $\La\ca Y$.  Assume that $\Theta: L^\infty(X)\rtimes \G \ra L^\infty(Y)\rtimes \Lambda$ is a $\ast$-isomorphism such that $\Theta(L(\G))=L(\La)$ and $\Theta (L^\infty(X_0)\rtimes \G)= L^\infty(Y_0)\rtimes \La$. Then there exists $y \in \mathcal U(L(\La)$, $\omega: \G \rightarrow \mathbb{T}$ a character, and $\delta: \G \rightarrow \La$ a group isomorphism such that $y\Theta(L^{\infty}(X))y^{\ast}=L^{\infty}(Y)$, and for all $a \in L^{\infty}(X), \g \in \G$, we have $$\Theta(au_{\g})=\omega(\g)\Theta(a) y^{\ast}v_{\delta(\g)}y.$$ In particular, we have $y \Theta (\sigma_\g(a))y^*= \alpha_{\delta(\g)}(y\Theta(a)y^*)$ and hence $\G\ca X$ and $\La\ca Y$ are conjugate. 
\vskip 0.02in	
	\noindent Here $\{u_\g\}_{\g\in \G}$  and $\{v_\la \}_{\la\in \La}$ are the canonical group unitaries implementing the actions in $L^\infty(X)\rtimes \G$ and $L^\infty(Y)\rtimes \La$, respectively.
\end{cor}

\begin{proof} Since $\G\ca X$ is mixing then by so is $\La \ca Y$. In particular the extensions $\G \ca (L^\infty(X_0)\subset L^\infty(X))$ and  $\G \ca (L^\infty(X_0)\subset L^\infty(X))$ are mixing. Since $\Theta (L^\infty(X_0)\rtimes \G)= L^\infty(Y_0)\rtimes \La$ the conclusion follows using the same arguments as in the proof of Theorem \ref{main2}. \end{proof}

\vskip 0.05in

\noindent Following the terminology from \cite{PV09} a free ergodic action $\G\ca X$ is called \emph{$\mathcal C_{\rm gms}$-superrigid} if up to unitary conjugacy $L^\infty(X)\subset L^\infty(X)\rtimes \G =M$ is the only group measure space Cartan subalgebra of $M$. Over the last decade many classes of examples of such actions have been discovered via deformation/rigidity theory. For some concrete examples the reader is referred to \cite{OP07,CS11,CSU11,PV11,PV12,Io12,CIK13} and the survey \cite{Io18}. An immediate consequence of \cite[Theorem 5.1]{Po04} is that all weakly mixing $\mathcal C_{\rm gms}$-superrigid actions satisfy the statement of Theorem \ref{main2}. Using our Theorem \ref{3'} we obtain the following generalization

\begin{cor} Any $\mathcal C_{\rm gms}$-superrigid action $\G \ca X$ of any icc group $\G$ satisfies the statement of Theorem \ref{main2}. \end{cor}

\noindent In particular the generalized Neshveyev-St\o rmer rigidity holds for all action $\G\ca X$ of icc groups $\G$ that are: hyperbolic groups, \cite{PV12}, free products \cite{Io12} or finite step extensions of such groups \cite{CIK13}.  

\vskip 0.05in
\noindent At this point it is increasingly evident that all the above Neshveyev-St\o rmer type rigidity results were achieved by heavily exploiting, at the von Neumann algebra level, the natural tension between mixing and compactness properties for action. It would be interesting to understand whether such results could still be obtained only in the compact regime. Specifically, we would like to propose for study the following 

\begin{prob} If $\G$ is icc does every free ergodic profinite action $\G\ca X$ satisfy the statement of Theorem \ref{main2}?  \end{prob} 

\noindent While providing a complete answer to this question seems hard at the moment, one can show there are many aspects of $\G \ca X$ that are shared by $\La \ca Y$ through this equivalence (e.g.\ compactness, profiniteness, etc). In fact we have the following result.  

\begin{thm}\label{profnsprob}Let $\Gamma\ca X$ be a free ergodic action and let $\Lambda\ca Y$ be any action. Let $\Theta: L^\infty(X)\rtimes \G \ra L^\infty(Y)\rtimes \Lambda$ be a $\ast$-isomorphism such that $\Theta(L(\G))=L(\La)$. Then  the following hold
\begin{enumerate}
\item If $\G \ca X$ is (weakly) compact then $\La\ca Y$ is also (weakly) compact.
\item If $\G$ is icc and $\G \ca X$ is profinite then $\La \ca Y$ is also ergodic and profinite. Specifically, if $\G \ca X$ is the inverse limit of $\G\ca X_n$ with $X_n$ finite then  $\La \ca Y$ is the inverse limit of $\G\ca Y_n$ with $Y_n$ finite so that for every $n$ we have $\Theta(L^{\infty}(X_n)\rtimes \G )=L^{\infty}(Y_n)\rtimes \La$. In addition, the stabilizer $Stab_\La(Y_n)\lhd \La$ is normal and we have that $\G/Stab_\G(X_n)\cong \La/Stab_\La(Y_n)$ for all $n$. Finally, there exists a normal subgroup $\Sigma\lhd \La$ so that $L^\infty(Y)'\cap L^\infty(Y)\rtimes \La= L^\infty(Y)\rtimes \Sigma$.
\end{enumerate}
\end{thm}

\begin{proof} \emph{1.} As before we assume that $L^\infty(X)\rtimes \G = L^\infty (Y)\rtimes \La=M$ and $L(\G)=L(\La)$. Since $\G \ca X$ is compact, \cite[Theorem 6.10]{Io08a} implies that the quasinormalizer algebra satisfies  $\mathcal{QN}_M (L(\G))''=M$. Since canonically $\mathcal{QN}_M(L(\G))'' = \mathcal{QN}_M(L(\La))''$ then $\mathcal{QN}_M(L(\La))''=M$ which by  \cite[Theorem 6.10]{Io08a} again implies that $\La \ca Y$ is also compact. The statement on weak compactness follows from \cite[Proposition 3.2]{OP07}.  
\vskip 0.04in	

\noindent \emph{2.} Since $\G$ is icc then $\La$ is also icc. Hence $\La \ca Y$ is ergodic (otherwise $M$ will not be a factor).  Next we show that $\La \ca Y$ is profinite. As $\G\ca X$ is profinite, it is the inverse limit of ergodic actions $\G\ca X_n$ on finite spaces. Thus $A_n=L^\infty(X_n)$ form  a tower of finite dimensional abelian $\G$-invariant subalgebras $A_0\subset...\subset A_n\subset A_{n+1}\subset...\subset L^\infty(X)$ such that $\overline{\cup_n A_n}^{SOT}=L^\infty(X)$. Moreover $\G\ca A_n$ is transitive for every $n$. Since $L^\infty(X_0)\rtimes \G=L^\infty(Y_0)\rtimes \La$ and $L(\G)=L(\La)$ using Theorem \ref{cptint} for every $n$ one can find a $\La$-invariant subalgebra $B_n \subset L^\infty(Y_0)$ such that $A_n\rtimes \G =B_n\rtimes \La$. Factoriality of $A_n \rtimes \G$ and $\La$ being icc imply that the action $\La \ca B_n$ is ergodic. Since $L(\G)\subseteq A_n \rtimes \G$ is a finite index inclusion of II$_1$ factors so is $L(\La)\subseteq B_n \rtimes \La$. Using Lemma \ref{finindex1} we get that $B_n$ is finite dimensional and the action $\La \ca B_n$ is transitive. One can easily check that $B_0 \subset ...\subset B_n\subset B_{n+1}\subset ... \subset L^\infty(Y_0)$ and also $\overline{\cup_n B_n}^{SOT}= L^\infty(Y_0)$. Thus there exist factors $\La \ca Y_n$ of $\La \ca Y$ with $Y_n$ finite such that $\La \ca Y$ is the inverse limit of $\La\ca Y_n$. 
 \vskip 0.04in
\noindent Denote by $\{p^n_i \,|\, 1\leq i\leq k_n\} = At(B_n)$. Since $Stab_\G (q)$ is assumed normal in $\G$ for every $q\in At(A_n)$ it follows from Proposition \ref{basicconstr1} that $Stab_\La(p^n_i)$ is normal in $\La$ for every $i$. Moreover, since the action $\La \ca B_n$ is transitive one can easily see that we actually have $Stab_\La(p^n_i)=Stab_\La(B_n)$ for all $1\leq i\leq k_n$. Finally, by Proposition \ref{basicconstr1} we also have that $\G/Stab_\G(X_n)\cong \La/Stab_\La(Y_n)$ for all $n$.   
 \vskip 0.04in

 \noindent In the remaining part we describe the relative commutant $L^\infty(Y)'\cap M$. So fix $b\in L^\infty(Y)'\cap M$ and consider its Fourier decomposition $b=\sum_{\la \in \La}b_\la v_\la$. Since $b$ commutes with $L^\infty(Y)$ we get that $yb_\la =\alpha_\la (y) b_\la$ for all $\la \in \La$ and $y\in L^\infty(Y)$. Letting $e_\la$ be the support projection of $b_\la b^*_\la$ this further implies that for all $y\in L^\infty(Y_0)$ and $\la \in \La$ we have  \begin{equation}\label{4.1.1}ye_\la =\alpha_\la (y) e_\la.
 \end{equation}
 
\noindent Fix $\la$ such that $b_\la\neq 0$ (and hence $e_\la\neq 0$). Denote by $e^n_\la:=E_{B_n}(e_\la)$ and applying the conditional expectation $E_{B_n}$ in  \eqref{4.1.1}, for all $y\in B_n$ we have \begin{equation}\label{4.1.2}ye^n_\la =\alpha_\la (y) e^n_\la.
 \end{equation}  Since $e_\la\neq 0$ then $e^n_\la\neq 0$ and hence there is $p^n_i\in At(B_n)$ satisfying $e^n_\la p^n_i= cp_i^n$ for some scalar $c>0$. Multiplying \eqref{4.1.2} by $c^{-1}p^n_i$ we get $yp^n_i =\alpha_\la (y) p^n_i$ for all $y\in B_n$. This entails that $\alpha_\la(p^n_i)=p^n_i$ and hence $\la\in Stab_\La(p^n_i)=Stab_\La(B_n)$. Altogether, we have shown that for every $\la$ with $b_\la\neq 0$ we have $\la\in Stab_\La(B_n)$. Applying this for every $n$ we conclude that $\la \in \cap_n Stab_\La (B_n)=:\Sigma$. In particular $b\in L^{\infty}(Y_0)\rtimes \Sigma$ and hence $L^\infty(Y)'\cap M \subseteq L^{\infty}(Y)\rtimes \Sigma$. Since the reverse containment canonically holds we get  $L^\infty(Y)'\cap M = L^{\infty}(Y)\rtimes \Sigma$. As $Stab_\La(B_n)$'s are normal in $\La$ then $\Sigma$ is also normal in $\Lambda$.   \end{proof}	

\noindent These results can be used to produce additional examples of actions satisfying the statement of Theorem \ref{main2}. For example part 1.\ of the previous theorem in combination with Theorem \ref{3'} and \cite[Theorem 4.16]{BIP18}, \cite[Theorem 5.1]{CPS12} shows that any free ergodic weakly compact action $\G \ca X$ satisfies the generalized Neshveyev-St\o rmer rigidity whenever $\G$ is an icc group in one of the following classes

\begin{enumerate}
\item $\G$ is any properly proximal group \cite[Definition4.1]{BIP18}, in particular when $\G=PSL_n(\mathbb Z)$, $n\geq 2$ or any $\G$ that admits a proper array into a nonamenable representation (see \cite[Definition 2.1]{CS11}). In fact the latter also follows by using the results in \cite{CS11,CSU11}; 
\item $\G= H\wr G$ is a wreath product where $H$ is nontrivial abelian  and $G$ nonamenable \cite{CPS12}.      
\end{enumerate}

\section{Some applications to strong rigidity results in von Neumann algebras and orbit equivalence}

\begin{thm}\label{virtualequal} Let $\G$ and $\La$ be icc property (T) groups. Let $\G \ca X = \lim X_n$ be a free ergodic profinite action and let $\La \ca Y$ be a free ergodic compact action. Assume that $\Theta: L^\infty(X)\rtimes \G \ra L^\infty(Y)\rtimes \La$ is a $\ast$-isomorphism. Then $\La \ca Y=\lim Y_n$ is a profinite action. Moreover there exists $l\in \mathbb N$ and a unitary $w\in L^\infty(Y)\rtimes \La$ such that $\Theta(L^\infty(X_{k+l})\rtimes \G)= w(L^\infty(Y_{k+1})\rtimes \La)w^*$ for every positive integer $k$.
\end{thm}

\begin{proof} To simplify the notations let $A=L^\infty(X)$,  $B=L^\infty(Y)$ and notice that  $M=A\rtimes \G= B\rtimes \La$. Moreover if $L^\infty(X_n)=A_n$. then $A_n\subseteq A_{n+1}$ and  $A=\overline{\cup_n A_n}^{SOT}$. Also if $M_n = A_n\rtimes \G$ then $M_n\subseteq M_{n+1}$ and $M=\overline{\cup_n M_n}^{SOT}$.

\noindent  Since $L(\La)\subset M$ is rigid subalgebra and $M$ has Haagerup property relative to $L(\G)$ it follows that $L(\La)\prec_M L(\G)$. Hence one can find nonzero projections $p\in L(\La)$, $q\in L(\G)$ a nonzero partial isometry $v\in M$  and an injective $\ast$-homomorphism  $\phi: pL(\La)p\ra rL(\G)r$ satisfying $\phi(x)v=vx$ for all $x\in pL(\La)p$. Since $L(\La)\subset M$ is a irreducible subfactor we have that $v^*v=p$. Denoting by $Q=\phi(pL(\La) p)$ we also have that $r=vv^*\in Q'\cap qMq$. Letting $u\in M$ a unitary such that $u v^*v =v$ we have that \begin{equation}\label{5.1}
u pL(\La)pu^*=Q r.
\end{equation} 
Next we prove the following

\begin{claim}\label{finindexQ} $Q\subseteq qL(\G)q$ is a finite index subfactor.
\end{claim}
\noindent \emph{Proof of Claim \ref{finindexQ}}. Since $L(\G)\subset M$ is a rigid subalgebra and $M$ has Haagerup's property relative to $L(\La)$ we also have that $L(\G)\prec_M L(\La)$. Since $L(\La)$ is a factor this further entails that $L(\G) \prec_M upL(\La)pu^*=Qr$. Hence by Popa's intertwining techniques there exist finitely many $x_j,y_j\in M$ and $c>0$ such that $\sum_j \|E_{Qr }(x_j u y_j)\|_2^2\geq c$ for all $u\in \mathcal U(L(\G))$. Since $E_Q(r)=\tau_q(r)q$ then we have $E_{Qr}(x)= E_Q(r)^{-1} E_Q(qxq) r= \tau_q(r)^{-1}E_Q(qxq)r$ for all $x\in L(\G)$. Using this formula in the previous inequality we further get that $\sum_j \|E_{Q }(qx_j u y_jq)\|_2^2\geq c \tau_q(r)>0$. Approximating $x_j$ and $y_j$ with their Fourier decompositions one can find finitely many $a_i, b_i \in A$ and $\g_i, \de_i\in \G$ such that for all $u\in\mathcal U(L(\G))$ we have $\sum_i \|E_{Q }( qu_{\g_i} a_iu b_i u_{\delta_i}q)\|_2^2\geq \frac{c\tau_q(r)}{2}$. Using this together with $E_Q=E_Q\circ E_{qL(\G)q}$ we see that for all $\g\in \G$ we have  

\begin{equation*}\begin{split}
 \frac{c\tau_q(r)}{2}&\leq \sum^s_{i=1} \|E_{Q }(q u_{\g_i} a_iu_\g  b_i u_{\delta_i}q)\|_2^2=\sum^s_{i=1} \|E_{Q }( qu_{\g_i} a_i \sigma_\g (b_i) u_{\g \delta_i}q)\|_2^2\\
 &=\sum^s_{i=1} \|E_{Q }( q E_{L(\G)}(u_{\g_i} a_i \sigma_\g (b_i) u_{\g \delta_i})q)\|_2^2 =\sum^s_{i=1} |\tau(a_i \sigma_\g (b_i))|^2\|E_{Q }( q  u_{\g_i \g \delta_i}q)\|_2^2\\
 &\leq (\max_{i=1,s}\|a_i\|^2_\infty\|b_i\|^2_\infty)(\sum_i \|E_{Q }( q  u_{\g_i \g \delta_i}q)\|_2^2).
\end{split}\end{equation*}   
Thus letting $d= \frac{c\tau_q(r)}{ 2 \max_{i=1,s}\|a_i\|^2_\infty\|b_i\|^2_\infty}$ we have that $\sum_i \|E_{Q }( q  u_{\g_i \g \delta_i}q)\|_2^2\geq d>0$ for all $\g\in G$. Hence by Theorem \ref{corner} we get $L(\G)\prec_{L(\G)} Q$ and since $L(\G)$ is a II$_1$ factor we actually have $qL(\G)q\prec_{L(\G)} Q$ and hence $qL(\G)q\prec_{qL(\G)q} Q$. In particular this entails $[qL(\G)q:Q]<\infty$ and the claim follows. $\hfill\blacksquare$
\vskip 0.05in
\noindent Combining the Claim \ref{finindexQ} with \cite[Lemma 3.1]{Po02} it follows the inclusion $qL(\G)q'\cap qMq\subseteq Q'\cap qMq$ has finite Pimsner-Popa probabilistic index. Since $\G$ is icc and $\G \ca X$ is free it follows that $L(\G)'\cap M =\mathbb C$ and thus $qL(\G)q'\cap qMq =\mathbb C q$. Combining with the above we conclude that $Q'\cap qMq$ is a finite dimensional von Neumann algebra. Since $Q\subseteq qL(\G)q\subset qM_1q\subset ...\subset qM_n q\subset qM_{n+1}q\subset ...\subset qMq$ and $qMq=\overline{\cup_n qM_n q}^{SOT}$ one can check that  $Q'\cap qM_1q\subset...\subset  Q'\cap qM_n q\subset Q'\cap qM_{n+1}q\subset ...\subset Q'\cap qMq$ and also  $Q'\cap qMq=\overline{\cup_n Q'\cap qM_n q}^{SOT}$. Since $Q'\cap qMq$ is finite dimensional there must be a minimal integer $l$ so that $Q'\cap qM_lq= Q'\cap qMq$. In particular, we have $r\in qM_nq$ and by \eqref{5.1} we obtain $upL(\La)pu^*\subseteq M_l$. As $M_l$ is a factor one can find $w\in \mathcal U(M)$ so that $w L(\La)w^*\subseteq M_l$. 
\vskip 0.03in
\noindent Since the action $\La\ca B$ is compact the using Theorem \ref{cptint} there is a $\La$-invariant von Neumann subalgebra $B_1\subset B$ satisfying $w (B_1\rtimes \La) w^*= A_l \rtimes \G=M_l$. Since $L(\G)$ has property (T) and $[M_l: L(\G)]<\infty$ it follows that $M_l$ has property (T). Thus $B_1\rtimes \La$ is a factor with property (T) and as $B_1\rtimes \La$ has Haagerup  property relative to $L(\La)$ we conclude that $B_1\rtimes \La \prec L(\La)$ and hence by \cite[Proposition 2.3]{CD18} we have  $[B_1\rtimes \La :L(\La)]<\infty$. Hence by Lemma \ref{finindex1} $B_1$ is finite dimensional and the action $\La \ca B_1$ is transitive. Finally, using Theorem \ref{cptint} successively there exist a tower of $\La$-invariant finite dimensional abelian von Neumann subalgebras $B_1\subset ...\subset B_n\subset B_{n+1}\subset ...\subset B$ such that $\overline{\cup_{n\geq 1} B_n}^{SOT} =B$ and also $w(B_{k+1}\rtimes \La) w^*=A_{k+l} \rtimes \G$ for all $k\geq 0$. Thus there exists a sequence of factors $\La \ca Y_n$ of $\La \ca Y$ into finite probability spaces $Y_n$ such that $L^\infty(Y_n)=B_n$ for all $n\geq 1$. From the previous relations one can check $\La \ca Y$ is the inverse limit of $\La\ca  Y_n$ which gives the desired statement. \end{proof}

%\begin{cor} Let $\G \ca X$ and $\La \ca Y$ be free ergodic actions where $\G$ is icc. Assume that  $\Theta: L^\infty(X)\rtimes \G \ra L^\infty(Y)\rtimes \La$ is a $\ast$-isomorphism such that $\Theta(L^\infty(X))=L^\infty(Y)$. Assume that there exist finite factors $\G\ca X_0$ and $\La\ca Y_0$ and a unitary $u\in L^\infty(Y)\rtimes \La$ such that $\Theta (L^\infty(X_0 \rtimes \G))= u L^\infty(Y_o)\rtimes  \La u^*$. Then the actions $\G\ca X$ and $\La \ca Y$ are virtually conjugate.  \end{cor}
 
%\begin{proof} For simplicity we assume that $L^\infty(X)=L^\infty(Y)=A$ and also $M= A\rtimes \G =A\rtimes \La$. Since $\G \ca X$ is profinite there exists an increasing tower of $\G$-invariant, finite dimensional algebras $(A_n))n\subseteq A$ such that $\overline{\cup_n A_n}^{SOT}=A$.    

%\end{proof}

\noindent The von Neumann algebraic methods developed in the previous sections can be used effectively to derive the following version of Ioana's $OE$-superrigidity theorem \cite[Theorem A]{Io08} for profinite actions of icc groups. 
\begin{thm}\label{ioanaoe} Let $\G \ca X$ be a profinite free ergodic action of an icc property (T) group $\G$ and let $\La\ca Y$ be an arbitrary free ergodic action of an icc group $\La$. Assume that  $\Theta: L^\infty(X)\rtimes \G \ra L^\infty(Y)\rtimes \La$ is a $\ast$-isomorphism such that $\Theta(L^\infty(X))=L^\infty(Y)$. Then there exist projections $p\in L^\infty (X)$ and $q\in L^\infty (Y)$, a unitary $u\in \mathcal N_{L^\infty(Y)\rtimes \La}(L^\infty(Y))$ with $u\Theta(p)u^*=q$, normal subgroups $\G'\lhd \G$, $\La'\lhd \La$ with $[\G:\G']=[\La:\La']<\infty$, a character $\eta: \G'\ra \mathbb T$ and a group isomorphism $\delta: \G'\ra \La'$ such that for all $\g \in \G'$ and $a\in A$ we have $$ \Theta(ap u_\g )= \eta(\g) \Theta(ap)  u^* v_{\delta(\g)} u. $$ 

\noindent In particular the actions $\G\ca X$ and $\La \ca Y$ are virtually conjugate.  \end{thm}

\begin{proof} Suppressing $\Theta$ we can assume $L^\infty(X)=L^\infty (Y)=A$ and $A\rtimes \G=A\rtimes \La=M$. Since property (T) is an $OE$-invariant \cite[Corollary 1.4]{Fu99a} it follows that $\La$ is also a property (T) icc group. Since $\G\ca A$ is profinite then it is weakly compact in the sense of Ozawa-Popa and by \cite[Proposition 3.4]{OP07} it follows that $\La\ca B$ is also weakly compact. Since $\La$ has property (T) then \cite[Remark 6.4]{Io08} implies that $\La\ca B$ is compact. Thus using Theorem \ref{virtualequal}  there exist increasing towers of $\G$-invariant, finite dimensional algebras $(A_n)_n\subseteq A$ and $(B_n)_n \subseteq A$ such that $\overline{\cup_n A_n}^{SOT}=A=\overline{\cup_n B_n}^{SOT}$. Also there is a unitary $w\in M$ and an integer $s$ such that for all $k$ we have \begin{equation}
w(A_{s+k}\rtimes \G)w^*= B_k\rtimes \La.
\end{equation}  
Since $A_s\rtimes \G\subseteq A_{s+k}\rtimes \G$ is a finite index inclusion of II$_1$ factors then so is  $B_0\rtimes \La\subseteq B_{k}\rtimes \La$. Thus by Proposition \ref{basicconstr1} it follows that $B_k$ is finite dimensional. Since $\G_{k}:= {\rm Stab}_\G (A_{s+k})\lhd \G$ is a finite index normal subgroup so is  $\La_k:={\rm Stab}_\La(B_k)\lhd \La$.   Since $w\in M= \overline{\cup_k A_k\rtimes \G}^{SOT}$ is a unitary there exists a sequence $w_k\in\mathcal U(A_k\rtimes \G)$ such that $\|w-w_k\|_2\ra 0$ as $k\ra \infty$.

\noindent For the remaining part of the proof for every $m\geq k$ we will keep in mind the following diagram of inclusions

\begin{equation}\label{tower1}\begin{array}{ccc}

ww_k^* (A_{s+m}\rtimes \G) w_kw^* & =& B_m\rtimes \La\\
\cup&&\cup\\
ww_k^* (A_{s+k}\rtimes \G) w_kw^* & =& B_k\rtimes \La\\
\cup&&\cup\\
w(A_{s}\rtimes \G) w^* & =& B_0\rtimes \La
\end{array}
\end{equation}

\noindent Pick $k$ large enough such that $\|1-ww_k^*\|_2\leq 10^{-9}$. Denote by $At(A_{s+l})=\{a^i_l \,:\, 1\leq i\leq r_l\}$ and $At(B_{l})=\{b^j_l \,:\, 1\leq i\leq t_l\}$. Also we can assume without any loss of generality that $\dim B_0 \leq \dim A_s$ (hence $\dim B_k \leq \dim A_{s+k}$ for all $k$); in particular we have $\tau(b^j_l)\geq \tau(a^i_l)$.  Fix $1\leq i\leq r_k$ such that $\|a_k^i(1-ww_k^*)\|_2+ \|(1-ww_k^*)a^i_k\|_2 \leq 2\|a^i_k\|\|1-ww_k^*\|_2$. Hence if we denote by $\delta^i_k=\|a_k^i(1-ww_k^*)\|_2+ \|(1-ww_k^*)a^i_k\|_2 \|a_k^i\|_2^{-1}$ then we have that  \begin{equation}
\delta^i_k \leq 2\|1-ww_k^*\|_2<10^{-8}.
\end{equation}

\noindent With this notations at hand we show that

\begin{claim}\label{height1} There is a unique $1\leq j\leq t_k$ such that for every $\g\in \G_k$ one can find  $\la, \la'\in \La$ such that 
\begin{eqnarray}&&(1-2\delta_k^i) \|b^j_k\|^2_2 \leq Re \tau(a_k^i u_\g v_\la b^j_k ), \text{and}\label{ineq4}\\
&&(1-3\delta_k^i) \|b^{j}_k\|^2_2 \leq Re \tau(ww_k^*a_k^i u_\g w_kw^*v_{\la'} b^{j}_k )\label{ineq5}.\end{eqnarray}
\end{claim} 

\noindent \textit{Proof of Claim \ref{height1}.} Fix $\g\in \G_k$. By triangle inequality we have $\| a^i_k u_\g-ww_k^* a_k^i u_\g w_kw^*\|_2\leq \delta_k^i\|a_k^i\|_2 $. Applying the conditional expectation and using \eqref{tower1} we also have $\|E_{B_k\rtimes \La}( a^i_k u_\g) -ww_k^* a_k^i u_\g w_kw^*\|_2\leq \delta_k^i\|a_k^i\|_2 $. Then the triangle inequality further gives \begin{eqnarray}
&&\|a_k^i u_\g-E_{B_k\rtimes \La}( a^i_k u_\g)\|_2\leq 2\delta_k^i\|a_k^i\|_2 \label{ineq1}\\
&&\|a_k^i u_\g- E_{B_k\rtimes \La}(ww_k^* a_k^i u_\g w_kw^*)\|_2\leq 3\delta_k^i\|a_k^i\|_2.\label{ineq2}
\end{eqnarray}

\noindent By Lemma \ref{Dye} there exist orthogonal projections $e_\la \in A$ so that $\sum_\la e_\la =1$ and $u_\g =\sum_{\la\in \La}  e_\la v_\la$. This combined with \eqref{ineq1} yield 
\begin{equation*}\begin{split}
\sum_{\la\in \La} \|a^i_k e_\la -E_{B_k }(a_k^i u_\g v_{\la^{-1}})\|^2_2&= \|a_k^i u_\g-E_{B_k\rtimes \La}( a^i_k u_\g)\|^2_2  \leq 4(\delta_k^i)^2\|a_k^i\|_2^2=\sum_{\la\in \La} 4(\delta_k^i)^2 \|a^i_k e_\la\|^2_2.
\end{split}
\end{equation*}
\noindent Thus one can find $\la\in \La$ so that $a^i_ke_\la\neq 0$ and $\|a^i_k e_\la -E_{B_k }(a_k^i u_\g v_{\la^{-1}})\|_2\leq 2\delta_k^i \|a^i_k e_\la\|_2$. This inequality and basic calculations show that $2(1- 2 \delta^i_k)\|a^i_k e_\la\|^2_2 \leq (1- 4(\delta^i_k)^2)\|a^i_k e_\la\|^2_2 +\| E_{B_k }(a_k^i u_\g v_{\la^{-1}})\|^2_2\leq 2 Re \tau(a^i_k e_\la E_{B_k }(a_k^i u_\g v_{\la^{-1}}))$; thus  
\begin{equation}\label{ineq3}(1- 2 \delta^i_k)\|a^i_k e_\la\|^2_2 \leq  Re \tau(a^i_k e_\la E_{B_k }(a_k^i u_\g v_{\la^{-1}})).   
\end{equation}

\noindent Using the formulas $\sum_j b_k^j=1$ and $E_{B_k}(x)=\sum_j\tau(x b_k^j)\tau(b^j_k)^{-1} b^j_k$, relation \eqref{ineq3} implies that $(1- 2 \delta^i_k)\sum_j \|a^i_k e_\la b^j_k\|^2_2 \leq  Re \sum_j \tau(a^i_k e_\la b_k^j)\tau (a_k^i u_\g v_{\la^{-1}} b^j_k)\tau(b^j_k)^{-1}$. Hence there is a $j$ (that at this point may depend on $\g$!) so that $a^i_k e_\la b^j_k\neq 0$ and  $(1- 2 \delta^i_k)\|a^i_k e_\la b^j_k\|^2_2 \leq  Re \tau(a^i_k e_\la b_k^j)\tau (a_k^i u_\g v_{\la^{-1}} b^j_k)\tau(b^j_k)^{-1}$. Simplifying $\|a^i_k e_\la b^j_k\|^2_2= \tau(a^i_k e_\la b^j_k)\neq 0$ this further gives \begin{equation}\label{ineq7}(1-2\delta_k^i)\|b_k^j\|^2_2\leq Re \tau(a_k^i u_\g v_\la^{-1} b_k^j).
\end{equation} 
\vskip 0.03in
\noindent Proceeding in a similar manner  inequality \eqref{ineq2} implies there exist $\la'\in \La$ and $1\leq j'\leq r_k$ such that 
\begin{equation}\label{ineq10}(1-3\delta_k^i)\|b_k^{j'}\|^2_2\leq Re \tau(ww_k^*a_k^i u_\g w_kw^*v_{\la'} b_k^{j'}).
\end{equation}  
\vskip 0.03in
\noindent To finish the proof it suffices to argue that $j=j'$ and $j$ is unique (hence does not depend on $\g$).  Since $\tau(a^i_k u_\g v_{\la^{-1}} b_k^j)= \tau(a^i_k E_A( u_\g v_{\la^{-1}}) b_k^j)= \tau(a^i_k e_\la b_k^j)$ then \eqref{ineq7} implies 
$\|b_k^j-a_k^i\|^2_2= \tau(b_k^j + a_k^i-2 b_k^ja_k^i)\leq \tau(b_k^j) + \tau(a_k^i)-2 \tau(b_k^j e_\la a_k^i)\leq 4\delta_k^i\tau(b_k^j)$ and hence
 \begin{equation}\label{ineq8}
\|b_k^j-a_k^i\|_2\leq 2(\delta_k^i)^{\frac{1}{2}}\|b_k^j\|_2.
\end{equation}
\noindent By triangle inequality this also yields \begin{equation}\label{ineq9}
\|b_k^j-ww_k^*a_k^i w_kw^*\|_2\leq (2(\delta_k^i)^{\frac{1}{2}}+ \delta_k^i)\|b_k^j\|_2.
\end{equation}  
Then Cauchy-Schwarz inequality in combination with \eqref{ineq10} and \eqref{ineq9} show that 
\begin{equation*}
\begin{split}
 \|b_k^{j'} b_k^{j} \|_2  &\geq  Re \tau( b_k^j ww_k^* u_\g w_kw^*v_{\la'}b_k^{j'}) \\
& \geq Re \tau(  ww_k^* a_k^i u_\g w_kw^*v_{\la'}b_k^{j'})- \|b_k^j-ww_k^*a_k^i w_kw^*\|_2 \|b_k^{j'}\|_2\\
& \geq (1-4\delta_k^i - 2(\delta_k^i)^{\frac{1}{2}})\|b_k^j\|^2_2.
\end{split}
\end{equation*}

\noindent As $b^j_k$'s are orthogonal this forces that $j=j'$. Uniqueness of $j$ (hence independence of $\g$) follows from \eqref{ineq8}. $\hfill\blacksquare$

\noindent Next we show the following 
\begin{claim}\label{height2} There exist $1\leq j\leq t_k$ and a unitary $s_k^j\in B_k\rtimes \La$ such that $s_k^j ww_k^* a_k^iw_kw^* (s_k^j)^* =c_k^j\leq p_k^j$  and 
\begin{equation}\label{equalcorner}s_k^j ww_k^* a_k^i (A_{s+k} \rtimes \G)a_k^iw_kw^* (s_k^j)^*= c_k^j (B_k\rtimes \La) c_k^j
\end{equation} and  for every  $\g\in \G_k$ there is  $\la'\in \La_k$ satisfying 
\begin{eqnarray}%&&(1-2\delta_k^i) \|b^j_k\|^2_2 \leq  Re \tau(a_k^i u_\g v_\la b^j_k ), \text{and}\label{ineq4}\\
&&(1-3\delta_k^i -36(\delta_k^i)^{\frac{1}{4}}) \|b^{j}_k\|^2_2 \leq  Re \tau(s_k^jww_k^*a_k^i u_\g w_kw^*(s_k^j)^*v_{\la'} b^{j}_k )\label{ineq11}.\end{eqnarray}
\end{claim} 
\noindent \textit{Proof of Claim \ref{height2}.} As $\tau(a_k^i)=\tau(ww_k^* a_k^iw_kw^*)\leq \tau(p^j_k)$ there is a subprojection $c^j_k\in B_k\rtimes \La$ of  $b_k^j$ that is equivalent (in $B_k\rtimes \La$) to $ww_k^* a_k^iw_kw^*$. By \cite[Lemma 4.1]{Co76} one can find a unitary $s_k^j\in B_k\rtimes \La$ satisfying $s_k^j ww^*_ka_k^i w_kw^* (s_k^j)^*=c_k^i$, $[s_k^i,|ww^*_ka_k^i w_kw^*-c_k^i|]=0$ and $|s_k^j-1|\leq 3|ww^*_ka_k^i w_kw^*-c_k^i|$. 

\noindent Using \eqref{ineq9} we see that \begin{equation*}\begin{split}\|ww^*_ka_k^i w_kw^*-c_k^i\|^2_2&\leq 2(2(\delta_k^i)^{\frac{1}{2}}+ \delta_k^i)^2\|b_k^j\|^2_2+2\|b_k^j-c_k^j\|^2_2=2(2(\delta_k^i)^{\frac{1}{2}}+ \delta_k^i)^2\|b_k^j\|^2_2+2\tau(b_k^j-c_k^j)\\
& = (2(2(\delta_k^i)^{\frac{1}{2}}+ \delta_k^i)^2+2)\|b_k^j\|^2_2-2\|a_k^i\|_2^2 \\
& \leq ((2(2(\delta_k^i)^{\frac{1}{2}}+ \delta_k^i)^2+2)-2 (1-2(\delta_k^i)^{\frac{1}{2}})^2) \|b_k^j\|_2^2\leq 18 (\delta_k^i)^{\frac{1}{2}}\|b_k^j\|_2^2.\end{split}\end{equation*}
\noindent Combining with the previous inequality we get $\|s_k^j-1\|_2\leq 18 (\delta_k^i)^{\frac{1}{4}}\|b_k^j\|_2$. In turn this together with Cauchy-Schwarz inequality and \eqref{ineq5} show that for every $\g\in \G_k$ there is $\la'\in \La$ such that

\begin{equation*}\begin{split}Re \tau ( s_k^j ww_k^* u_\g a_k^iw_k^*w (s_k^j)^* v_{\la'} b_k^j) &\geq  Re \tau ( ww_k^* u_\g a_k^i w_k^*w  v_{\la'} b_k^j) - 2 \|1-s_k^j\|_2 \|b_k^j\|_2 \\
&\geq  (1-3\delta_k^i- 36(\delta_k^i)^{\frac{1}{4}})\|b_k^j\|_2^2.
\end{split}
\end{equation*}

\noindent This shows \eqref{ineq11}. Also since $\|c_k^j v_{\la'} b_k^j\|_2\geq Re \tau ( s_k^j ww_k^* u_\g a_k^iw_k^*w (s_k^j)^* v_{\la'} b_k^j)$ the above inequality also shows that $\la'\in \La_k$. The rest of the statement follows from the previous relations. $\hfill\blacksquare$

\vskip 0.05in
\noindent  Since $a_k^i (A_{s+k}\rtimes \G) a_k^i= L( \G_k) a_k^i$ and $b_k^j( B_k\rtimes \La) b_k^j= L(\La_k) b_k^j$ then \eqref{equalcorner} of Claim \ref{height2} implies that $s_k^j ww_k^*  L( \G_k ) a_k^i w_kw^* (s_k^j)^* = c_k^j L(\La_k) b_k^j c_k^j\subseteq L(\La_k)b_k^j$. Since $\La_k$ and $\G_k$ are icc groups we see that the conditions (1) and (2) in \cite[Theorem 4.1]{KV15} are satisfied, where $\mathcal G =s_k^j ww_k^*  \G_k a_k^i w_kw^* (s_k^j)^*$. Also \eqref{ineq11} shows that (3) in \cite[Theorem 4.1]{KV15} is also satisfied. Therefore using the conclusion of that theorem we get that $c_k^j=b_k^j$.
\vskip 0.03in
\noindent In conclusion we have that $s_k^j ww_k^*  L( \G_k ) a_k^i w_kw^* (s_k^j)^* =  L(\La_k) b_k^j$. Notice we also have $s_k^i ww_k^* (A\rtimes \G_k  )a_k^i w_kw^* (s_k^i)^*=s_k^i ww_k^* a_k^i (A\rtimes \G  )a_k^i w_kw^* (s_k^i)^*=    b_k^j (A\rtimes \G) b_k^j = (A\rtimes \G_k) b_k^j$. Let $z\in \mathcal N_M(A)$ such that $z^*z= a_k^i$ and $z^*z= b^j_k$ and denote by $y= b_k^j z w_kw^* (s_k^j)^*$ one can check that $y$ is a unitary in $(A\rtimes \La_k) b_k^j$ satisfying $y(s_k^j ww_k^* A a_k^i w_k w^* (s_k^j)^*)y^*= Ab_k^j$. Thus applying Theorem \ref{3'} (working with the algebra $(A\rtimes \La_k) b_k^j$) we get the desired conclusion by letting $p=a_k^i$ $q=b_k^j$ etc.\end{proof}

\noindent {\bf Final remarks.} We notice that \eqref{ineq11} can be used directly to show that $\G_k$ is isomorphic to finite index subgroup of $\La_k$. It is plausible that one can exploit this further and show the conclusion directly, without appealing to the results in \cite{IPV10,KV15}. %Also we point out if in the last part of the proof of Theorem \ref{ioanaoe} one uses the full strength of \cite[Theorem 4.1]{KV15} then one can circumvent the (implicit) use of our Theorem \ref{2'}.

\section*{Acknowledgments}  The authors are grateful to Adrian Ioana and Jesse Peterson for many helpful discussions related to this project. The authors are extremely grateful to Yuhei Suzuki for carefully reading a first draft of this paper, for his helpful comments and suggestions, and for correcting numerous typos and minor inaccuracies. The authors are also grateful to Rahel Brugger for her helpful comments on our paper, and for correcting a few typos.
 The second author would like to thank Vaughan Jones for several suggestions and comments regarding the results of this paper. The second author would also like to thank Krishnendu Khan and Pieter Spaas for stimulating conversations regarding the contents of this paper. The first author was partially supported by NSF Grant DMS \#1600688.

%\small
%\bibliographystyle{amsalpha}
%\bibliography{refpop}

%\begin{thebibliography}{99} % don't worry about the 99

%\bibitem{ExamplePaper}
%E.~Edelman, ``The probability that a random real Gaussian matrix has $k$ real eigenvalues, related distributions, and the circular law'', {\em J.~Multivariate~Anal.} {\bf 60} (1997), no.~2, 202--232.

%\bibitem{ExampleBook}
%H.~L.~Montgomery and R.~C.~Vaughan, {\em Multiplicative Number Theory I: Classical Theory}, Cambridge University Press (2007).

%\end{thebibliography}

%\noindent
%\textsc{Department of Mathematics, Vanderbilt University, 1326 Stevenson Center, Nashville, TN 37240, U.S.A.}

\noindent
\textsc{Department of Mathematics, The University of Iowa, 14 MacLean Hall, Iowa City, IA 52242, U.S.A.}\\
\email {ionut-chifan@uiowa.edu} \\
\email{sayan-das@uiowa.edu}

\end{document}